\documentclass[a4paper]{article}
\UseRawInputEncoding
\usepackage{amsthm}

\newtheorem{defi}{Definition}[section]
\newtheorem{thm}[defi]{Theorem}%[section]
\newtheorem{lem}[defi]{Lemma}%[section]
\newtheorem{prop}[defi]{Proposition}%[section]

\usepackage[toc,page]{appendix}
\usepackage{url}
\usepackage{pgfplots}
\usepackage{tikz}
\usepackage{amsmath}
\usepackage{mathtools}
\usepackage[a4paper, left=3cm, right=3cm, top=4cm]{geometry}
\usepackage{amssymb}
\usepackage{amsfonts}
\usepackage{scrlayer-scrpage}
\usepackage{hyperref}
\usepackage{accents}
\usepackage{ulem}
\clearscrheadfoot
\newcommand{\R}{\mathbb{R}}

\newcommand{\Z}{\mathbb{Z}}

\def\XXint#1#2#3{{\setbox0=\hbox{$#1{#2#3}{\int}$ }
\vcenter{\hbox{$#2#3$ }}\kern-.6\wd0}}

\AtBeginDocument{
  \let\div\relax
  \DeclareMathOperator{\div}{div}
}

\usepackage{setspace}
\makeatletter
\newcommand{\MSonehalfspacing}{%
  \setstretch{1.44}%  default
  \ifcase \@ptsize \relax % 10pt
    \setstretch {1.448}%
  \or % 11pt
    \setstretch {1.399}%
  \or % 12pt
    \setstretch {1.433}%
  \fi
}
\newcommand{\MSdoublespacing}{%
  \setstretch {1.92}%  default
  \ifcase \@ptsize \relax % 10pt
    \setstretch {1.936}%
  \or % 11pt
    \setstretch {1.866}%
  \or % 12pt
    \setstretch {1.902}%
  \fi
}
\makeatother

\ohead{\pagemark}

\begin{document}

	\title{Half-Harmonic Gradient Flow: Aspects of a Non-Local Geometric PDE}
	
	\author{Jerome Wettstein}
\maketitle
\date{ }
\begin{abstract}
	The goal of this paper is to discuss some of the results in \cite{wettsteinsphere} and \cite{wettstein} and expand upon the work there by proving a global weak existence result as well as a first bubbling analysis in finite time. In addition, an alternative local existence proof is presented based on a fixed-point argument. This leads to two possible outlooks until a conjecture by Sire, Wei and Zheng is settled, see \cite{sireweizheng}: Either there always exists a global smooth solution (thus the solution constructed here is actually as regular as desired) or finite-time bubbling may occur in a similar way as for the harmonic gradient flow.

	In addition, in this paper, we give a brief summary of gradient flows, in particular the harmonic and half-harmonic one and we draw similarities between the two cases. For clarity, we restrict to the case of spherical target manifolds, but our entire discussion extends after taking care of technical details to arbitrary closed target manifolds.
\end{abstract}

\medskip

\tableofcontents 

\section{Introduction}

	Among the most prominent partial differential equations is the harmonic map equation. Its relevance derives from the way they emerge naturally (these maps arrive for example in theoretical physics they appear as instantons or their close connection to curvature) as well as the fact that the associated PDE, especially in the critical realm where the domain is two-dimensional, is related to the creation of powerful techniques such as H\'elein's moving frames method, Uhlenbeck's Coulomb gauge construction and bubbling techniques. The harmonic map equation has also inspired the introduction of fractional harmonic maps by Da Lio-Rivi\`ere \cite{dalioriv}, \cite{dalioriv2} and the corresponding regularity theory (\cite{dalioriv2}, \cite{schikorradaliocomment}, \cite{daliocomment1}, \cite{daliocomment2}, \cite{dalioschikorra}, \cite{mazoschi}) and bubbling analysis (\cite{daliobubbles}, \cite{daliobubbles2}) have led to generalisations of various ideas from the local world, such as Wente-type estimates, gauge techniques, Pohozaev identities and many others in the fractional world.\\
	
	Let us take a step back and recall the main definitions that we shall be using. For the moment, let us take $(M, g)$, $(N, \gamma)$ to be arbitrary closed Riemannian manifolds. Using them, one is naturally led to define the following Dirichlet energy for maps $u: M \to N$:
	\begin{equation}
	\label{gendirichletenergy}
		E(u) := \frac{1}{2} \int_{M} g^{\alpha \beta}(x) \gamma_{ij}(u(x)) \frac{\partial u^{i}}{\partial x_{\alpha}}(x) \frac{\partial u^{j}}{\partial x_{\beta}}(x) dx,
	\end{equation}
	where we assume for convenience that $M, N$ are embedded submanifolds of the Euclidean space and use Einstein's summation convention. Naturality of this definition becomes apparent if one realises that given $M = T^{m}, N = \R^n$ with the corresponding natural Riemannian metrics, the energy simplifies to the usual Dirichlet energy:
	\begin{equation}
	\label{gendirichletenergy2}
		E(u) = \frac{1}{2} \int_{T^{m}} | \nabla u |^2 dx
	\end{equation}
	Critical points of \eqref{gendirichletenergy2} satisfy the equation:
	$$- \Delta u = 0,$$
	which immediately, by means of elliptic regularity, yields $u \in C^{\infty}(T^n)$. A natural question thus pertains to the regularity of critical points of \eqref{gendirichletenergy}, where one may interpret the condition that $u$ takes values in $N$ as a Lagrange-multiplier. This also leads us to the first key definition:
	\begin{defi}
		A map $u \in H^{1}(M;N)$ is called \textit{harmonic}, if and only if it is a critical point of the energy function $E$ as defined in \eqref{gendirichletenergy} among competitors in $H^{1}(M;N)$.
	\end{defi}
	Here, we recall that $H^{1}(M;N)$ consists of all functions $u \in H^{1}(M;\R^n)$ such that $u(x) \in N$ for almost all $x \in M$. Therefore, the choice of space is appropriate for both the energy functional $E$ as well as the condition of $u$ assuming values in $N$. As usual with variational problems, computing the Euler-Lagrange equation is the first step towards regularity results and in the case of harmonic maps, we get:
	\begin{equation}
	\label{harmoniceulerlagrange}
		u \text{ is harmonic } \Leftrightarrow (-\Delta)_{M} u \perp T_u N \text{ in } \mathcal{D}'(M) \Leftrightarrow (-\Delta)_{M} u = A(u)(\nabla u, \nabla u),
	\end{equation}
	where $A$ denotes the second fundamental form of $N$ and $\Delta_M$ being the Laplace-Beltrami operator of $M$. This already highlights the intimate connection of harmonic maps to curvature and geometry of the manifolds and indeed, existence may be tied to conditions on the curvature, see for example \cite{schoenyau}. A particularily striking feature of the PDE \eqref{harmoniceulerlagrange} is the following quadratic structure:
	$$| A(u)( \nabla u, \nabla u) | \leq | \nabla u |^2,$$
	which immediately singles out the case of $M$ being two-dimensional as being critical. Indeed, in this case, if $u \in H^1$, then the RHS of \eqref{harmoniceulerlagrange} is in $L^{1}$. In this case, one merely deduces $\nabla u \in L^{2,\infty}$, so we have regularity of the same homogeneity as before and standard bootstrapping techniques are not immediately available. What is even worse, there is no general regularity theory for solutions of such equations with smooth quadratic non-linearity as can be seen from:
	$$-\Delta u = | \nabla u |^2$$
	In this case, by using $v := e^{u}$, one may immediately construct counterexamples by using the observation that the above PDE is equivalent to $-\Delta v = 0$ and then considering suitable versions of the fundamental solution, see \cite{riv3} for details.\\
	
	Fortunately, the case of the harmonic map equation \eqref{harmoniceulerlagrange} behaves much better due to its geometric structure. Namely, by combined efforts of various authors such as B\'ethuel, Coron, H\'elein, Shatah and Rivi\`ere, just to name a few, we know that harmonic maps are always regular. A common feature among proof of regularity is the use of compensation results based on properties of $2$D-jacobians as summarised by Wente's estimate and later extended to arbitrary div-curl quantities and determinants by Coifman-Lions-Meyer-Semmes which we state here for the reader's convenience:
	\begin{prop}
		Let $r > 0$ be arbitrary and $1 \leq p < 2$ and $B_r(0) \subset \R^2$ the ball of radius $r$ around $0$. If $u \in W^{1,p}(B_r(0))$ and $a,b \in W^{1,2}(B_r(0))$ are such that:
		$$-\Delta u = \partial_x a \partial_y b - \partial_x b \partial_y a,$$
		and $u$ has vanishing trace on $\partial B_r(0)$, then $u$ is actually continuous and the following estimate holds:
		\begin{equation}
		\label{wenteest}
			\| u \|_{L^{\infty}} + \| \nabla u \|_{L^{2,1}} + \| \nabla^2 u \|_{L^1} \lesssim \| \nabla a \|_{L^2} \| \nabla b \|_{L^2}
		\end{equation}
	\end{prop}
	A similar result continues to hold for the RHS being a product of a divergence-free and a rotation-free vector field and in fact, the underlying key result is a Hardy-regularity estimate for the RHS\footnote{In contrast to $L^{1}$, the Hardy space $\mathcal{H}^{1}$ is well-behaved with respect to Cald\'eron-Zygmund operators and thus elliptic regularity results apply in this case.}. One may wonder how this estimate helps us treat equations like \eqref{harmoniceulerlagrange}. The idea is that either employing H\'elein's moving frames or Rivi\`ere's change of gauge approach, one may reveal a hidden structure of jacobians in the harmonic map equation. This is best illustrated in the case $M = T^{2}, N = S^{n-1}$, where the harmonic map equation reads:
	\begin{equation}
	\label{harmonicmapsphereeq}
		- \Delta u = u | \nabla u |^2
	\end{equation}
	Shatah observed that the equation \eqref{harmonicmapsphereeq} is actually equivalent to the following conservation laws:
	$$\forall i,j \in \{ 1, \ldots n \}: \div \left( u_i \nabla u_j - u_j \nabla u_i \right) = 0$$
	Thus, we have:
	\begin{align}
		-\Delta u_i 	&= \sum_j u_i \nabla u_j \cdot \nabla u_j \notag \\
					&= \sum_j \left( u_i \nabla u_j - u_j \nabla u_i \right) \cdot \nabla u_j + u_j \nabla u_i \cdot \nabla u_j \notag \\
					&= \sum_j \left( u_i \nabla u_j - u_j \nabla u_i \right) \cdot \nabla u_j,
	\end{align}
	where in the last line we used the observation that $\sum_j u_j \nabla u_j = 0$, as $u$ takes values in $S^{n-1}$ and thus this may be interpreted as a vector of scalar products between $u$, which belongs to the normal space of $S^{n-1}$, and partial derivatives of $u$, which are in the tangent space of $S^{n-1}$, thus these vanish. Keeping Shatah's conservation laws in mind, the structure of Wente's estimates is now uncovered and by localising, splitting $u$ into harmonic and zero-boundary value parts, one may establish a suitable Morrey decrease and this ultimately allows, by virtue of Adams embedding, to conclude that $\nabla u \in L^{p}_{loc}$ for some $p > 2$ (see \cite{riv2}). The remaining part of the regularity proof is then a bootstrap argument.\\
	
	A generalisation of the notion of harmonic maps was developed later by Da Lio and Rivi\`ere in \cite{dalioriv} following the spirit above. Namely, taking $N$ to be an embedded submanifold of $\R^n$ with the induced Riemannian structure, by using the $s$-Dirichlet energy defined by:
	$$E_{s}(u) := \int_{S^1} | (-\Delta)^{s/2} u |^2 dx, \quad \forall u \in H^{s}(S^1,N),$$
	one may define:
	\begin{defi}
	\label{deffracharmonic}
		A map $u \in H^{s}(S^1;N)$ is called \textit{$s$-harmonic}, if and only if it is a critical point of the energy function $E_s$ among variations in the space $H^{s}(S^1;N)$.
	\end{defi}
	Fractional Laplacians $(-\Delta)^{s/2}$ may be defined by Fourier multipliers or using principal value integrals, we refer to \cite{hitchhiker} for some exposition or the next subsection where both kinds of definitions are introduced. These maps are related to free-boundary minimal discs (\cite{dalioriv}, \cite{daliopigati}) and singular limits of Ginzburg-Landau approximations (\cite{millotsire}), so there are again interesting connections to geometry.
	
	There is no reason to use $S^1$ instead of $\R$ and most results are available in both cases. If $s = 1/2$, which will interest us particularily, $E_{1/2}$ is conformally invariant (under the trace of M\"obius transformations) and the stereographic projection enables us to switch between $\R$ and $S^1$ seemlessly. For the remainder of this paper, we shall restrict our attention to this case.\\
	
	As a very simple first case, if $N = \R$, then the Euler-Lagrange equation reads:
	$$(-\Delta)^{1/2} u = 0,$$
	which, similar to the harmonic case, immediately proves regularity. In general, the Euler-Lagrange equation has a very similar structure to \eqref{harmoniceulerlagrange}, namely it becomes:
	\begin{equation}
	\label{halfharmmapeq}
		u \text{ is } \frac{1}{2}\text{-harmonic } \Leftrightarrow (-\Delta)^{1/2} u \perp T_u N \text{ in } \mathcal{D}'(S^1)
	\end{equation}
	In \cite{dalioriv}, the equation above has been rewritten using Three-commutators to reveal compensation structures. In addition, in the case $N = S^{n-1}$, we may push the similarity with \eqref{harmonicmapsphereeq} further by noting the equivalence of \eqref{halfharmmapeq} with:
	\begin{equation}
	\label{halfharmsphere}
		(-\Delta)^{1/2} u = u | d_{1/2} u |^2,
	\end{equation}
	where we use the framework of fractional gradients as introduced in \cite{mazoschi}:
	$$d_{1/2} u(x,y) = \frac{u(x) - u(y)}{|x-y|^{1/2}}, \quad | d_{1/2} u|^2(x) := \int_{S^1} | d_{1/2} u(x,y) |^2 \frac{dy}{| x-y |}$$
	For completeness' sake, let us emphasise here that we use the natural distance $| x - y | = 2 | \sin(\frac{x-y}{2}) |$ on $S^1$. It should be noted that the fractional gradients are very natural and relate to the Gagliardo-Sobolev spaces and Bessel potential spaces, see also \cite{prats1}, \cite{prats2}.\\
	
	It is interesting to observe that the key features of \eqref{harmoniceulerlagrange} are also present in \eqref{halfharmmapeq}: Both possess a quadratic RHS and are critical equations not allowing for simple regularity by bootstrap techniques. Indeed, in \cite{wettstein}, we even establish a "curvature-like" formulation similar to \eqref{halfharmsphere} in general, but the structure is most easily recognizable in the case $N = S^{n-1}$.\\
	
	Regularity properties of half-harmonic maps have been studied extensively and we know that every such function is actually regular. The investigation was started by \cite{dalioriv}, \cite{dalioriv2} and later expanded by other authors in \cite{schikorradaliocomment} and \cite{mazoschi}, the proofs becoming increasingly similar to the local analogue in the case of the harmonic map equation. There are essentially two approaches which are indeed two sides of the same medal: Three commutators based on improved regularity of certain linear combinations of terms and non-local Wente estimates. For illustration, Three commutator estimates are in some sense various incarnations of estimates of operators such as:
	$$\mathcal{T}: L^{2}(\R; \R^{m}) \times \dot{H}^{1/2}(\R; \R^{m \times m}) \to \dot{H}^{-1/2}(\R; \R^{m}),$$
	defined by:
	$$\mathcal{T}(v,Q) := (-\Delta)^{1/4} (Qv) - Q (-\Delta)^{1/4} v + (-\Delta)^{1/4} Q \cdot v$$
	So this operator quantifies the failure of Leibniz' rule for the $1/4$-Laplacian. For instance, it is proven in \cite{dalioriv} that:
	$$\| \mathcal{T}(v,Q) \|_{\dot{H}^{-1/2}} \lesssim \| Q \|_{\dot{H}^{1/2}} \| v \|_{L^2}$$
	One should keep in mind that a-priori, each summand in $\mathcal{T}(v,Q)$ individually does not belong to $\dot{H}^{-1/2}$.
	
	To draw similarities with the harmonic map equation, we focus on the non-local Wente-type estimate for now:
	\begin{prop}
	\label{fracwente}
		Let $s \in (0,1)$ and $p \in (1,\infty)$. For $F \in L^{p}_{od}(\R \times \R)$ and $g \in \dot{W}^{s,p'}(\R)$, where $p'$ denotes the H\"older dual of $p$, we assume that $\div_s F = 0$. Then $F \cdot d_s g$ lies in the Hardy space $\mathcal{H}^{1}(\R)$ and we have the estimate:
		$$\| F \cdot d_s g \|_{\mathcal{H}^{1}(\R)} \lesssim \| F \|_{L^{p}_{od}(\R \times \R)} \cdot \| g \|_{\dot{W}^{s,p'}(\R)}.$$
	\end{prop}
	The general $s$-gradient is introduced in analogy to the $1/2$-gradient and we say that $\div_{s} F = 0$, if for all $\varphi \in C^{\infty}_{c}(\R)$:
	$$\int_{\R} \int_{\R} F(x,y) d_{s} \varphi(x,y) \frac{dy dx}{| x-y |} = 0.$$
	Lastly, we define:
	$$F \cdot d_s g (x) := \int_{\R} F(x,y) d_s g(x,y) \frac{dy}{| x-y |}.$$
	These notions and results also apply for $S^1$ as a domain and we refer to \cite{mazoschi} for details.\\
	
	Regularity is now analogous to the harmonic map case, up to taking care of tail estimates emerging due to the non-locality of the equations, we just give a outline in the case $N = S^{n-1}$: Shatah-like fractional conservation laws hold
	$$\forall i, j \in \{1, \ldots, n \}: \div_{1/2} \left( u_i d_{1/2} u_j - u_j d_{1/2} u_i \right) = 0,$$
	leading to a reformulation of \eqref{halfharmsphere} as:
	$$(-\Delta)^{1/2} u_i =  \sum_j \left( u_i d_{1/2} u_j - u_j d_{1/2} u_i \right) \cdot d_{1/2} u_j + \sum_j u_j d_{1/2} u_i \cdot d_{1/2} u_j,$$
	where, however, the last term is now not vanishing as the fractional gradient is not necessarily tangential to $S^{n-1}$. Nevertheless, using $\pi$ the closest point projection extended suitably and Taylor expansion, we see:
	$$u(x) - u(y) = \pi(u(x)) - \pi(u(y)) = d\pi(u(x)) \left( u(x) - u(y) \right) + R(u(x), u(y)),$$
	with $| R(u(x), u(y)) | \lesssim | u(x) - u(y) |^2$. Noting that $d\pi(u(x))$ is the tangent projection at $u(x)$, we may deduce that the remainder may be considered to be of lower order and thus does not obstruct the argument. Then, localising and applying Wente-type estimates, we arrive again at a Morrey decrease and, by Adams embedding, at higher integrability and thus H\"older regularity. Following \cite{daliopigati}, one may bootstrap this information to arrive at smoothness of solutions.\\
	
	The considerations above hopefully convinced the reader of the close connection between $1/2$-harmonic and harmonic equations, at least in the structure of their proof and, in some sense, their relation to the geometry of the target manifold. This relation is something that we exploited in the papers \cite{wettsteinsphere}, \cite{wettstein} to establish a theory for the half-harmonic gradient flow in analogy to \cite{struwe1} and obtain a uniqueness statement in the small energy realm by extending techniques introduced in \cite{riv}. Thus, we shall embark on a short survey of some results pertaining to the harmonic gradient flow.
	
	To start, let us present the main equation: We consider functions $u: [0,T[ \times M \to N$, $M, N$ being embedded submanifolds with the induced metrics and $T \in \R \cup \{ + \infty \}$ and would like to solve the following equation:
	\begin{equation}
	\label{harmgradfloweq}
		\partial_{t} u - \Delta_{M} u = A(u)( \nabla u, \nabla u),
	\end{equation}
	with initial value $u(0, \cdot) = u_0(\cdot) \in H^{1}(M;N)$. Its relevance derives from approximations of stationary solutions (i.e. harmonic maps) and questions pertaining to the homotopy of maps like whether any given map in $H^{1}(M;N)$ is homotopic to a harmonic one. The latter question obviously not being true, one may wonder what kind of convergence and regularity properties are to be expected in finite time as well as when $t \to + \infty$. We shall see, also in the case of the half-harmonic gradient flow, that bubbling may occur and especially in the harmonic flow, a variety of different types of bubbling (finite time, infinite time, reverse bubbling) may occur, see especially the extraordinarily well-suited approach to bubbling using the inner-outer gluing scheme. We recall that, rather informally, a bubble is a harmonic map which is created by energy concentration in smaller and smaller neighbourhoods of a point after blow up using rescalings. Outside of the blow up point, the flow will behave nicely and be smooth, but at the bubbling point itself, energy accumulates and results in the formation of a so-called bubble. We refer to \cite{struwe1} for more details.\\
	
	The harmonic gradient flow was first studied by Eells and Sampson culminating in an existence result, provided that the target manifold has non-positive sectional curvature. However, the first general result that applies independent of any geometric properties (such as sectional curvature) was given by Struwe in \cite{struwe1} for two-dimensional domains and extended to arbitrary domains in \cite{struwe2}. We state the result here in a special case, as it may be found for example in \cite{riv}:
		\begin{thm}
	\label{mainstruwe1}
		Let $u_0 \in H^{1}(T^2; S^{n-1})$. Then there exists a solution $u \in H^{1}(]0, +\infty[;L^{2}(T^2))$ of the harmonic gradient flow:
		\begin{equation}
		\label{harmonicflowstruwe1}
			\partial_{t} u - \Delta u = u | \nabla u |^2 \quad \text{ in } \mathcal{D}'(]0, T[ \times T^2), \quad \forall T > 0,
		\end{equation}
		together with the boundary conditions:
		\begin{equation}
		\label{harmonicflowstruwe1bound2}
			u(0,x) = u_{0}(x), \quad	\text{ for all } x \in T^2,
		\end{equation}
		and satisfying $E(u(t, \cdot)) \leq E(u_{0})$ for all times $t \geq 0$. The solution $u$ is regular on $]0,+\infty[ \times {T^2}$, except in a finite number of points $(t_{k}, x_{k})$, $k = 1, \ldots, K$, for some $K \in \mathbb{N}$. Additionally, $u$ is unique in the class $\mathcal{E} \subset H^{1}_{loc}([0,+\infty[ \times {T^2})$ defined by:
		$$\mathcal{E} := \Big{\{} u \ \Big{|}\ \exists m \in \mathbb{N}, \exists T_0 = 0 < T_1 < \ldots < T_{m} < \infty: u \in L^{2}([T_{i}, T_{i+1}[; W^{2,2}(T^2)), \forall i \leq m-1 \Big{\}}$$
		Finally, there exists a constant $C > 0$ independent of $u_0$, such that:
		$$K \leq C \cdot E(u_0)$$
	\end{thm}
	Theorem \ref{mainstruwe1} immediately addresses existence, regularity and bubbling questions as well as providing a first uniqueness result (at least for strong solutions, i.e. solutions with sufficient regularity to make sense of the equation in the $L^2$-sense). Interesting features include the fact that bubbling may occur, but indeed, no energy is lost. So the energy as a function of time is continuous up to the creation of bubbles which account for the jump down in energy completely. This also guarantees that the amount of bubbles that form is finite.
	
	The proof of Theorem \ref{mainstruwe1} relies on testing the equation \eqref{harmonicflowstruwe1} against itself or appropriate derivatives. As a very simple example, let us test \ref{mainstruwe1} against $\partial_{t} u$, then we find:
	\begin{equation}
		\int_{0}^{T} \int_{T^2} | \partial_{t} u |^2 dx dt + E(u(T)) - E(u_0) = \int_{0}^{T} \int_{T^2} | \partial_{t} u |^2 dx dt + \int_{0}^{T} \int_{T^2} \partial_{t} \left( \frac{1}{2} | \nabla u |^2 \right) dx dt = 0,
	\end{equation}
	since $\partial_t u$ is tangential to $S^{n-1}$, while $u$ is perpendicular to the tangent space. This shows for example energy decrease for regular solutions. Similar applications of such test functions lead to identities and estimates after absorption (using, for instance, quantification of absolute continuity of $L^4$-norms) depending only on energy concentration in balls and therefore, thanks to good control of the energy for a finite amount of time, one is able to prove uniform estimates for regular solutions (for example regarding the $H^2$-norm of $u(t)$ at a fixed time). By approximation, the result extends to arbitrary initial values, no matter their regularity. If the energy concentration is however too big, a bubble will form and by rescaling around bubbling points and extracting subsequences, one finds that bubbles $v: S^{2} \to N$ form which are indeed harmonic maps.\\
	
	There are several natural questions to ask from here: Firstly, one may wonder if bubbling in finite time is actually possible. Indeed, this question is answered in \cite{chang} by using explicit solutions in the corotational setting. Indeed, they construct subsolutions blowing up in finite time and prove that appropriate boundary conditions exist to transfer the blow up to a solution of the harmonic map flow.
	
	Another question pertains to whether the energy decay is necessary as an assumption. Indeed it is, without this no uniqueness statement is possible. For instance, in \cite{topping}, other kinds of blow ups violating the monotone decay of energy are constructed by using so-called reverse bubbling and prove existence of "non-physical" solutions. Furthermore, various types of blow ups may be considered using the inner-outer gluing scheme, studied by del Pino, Musso, Wei and many other authors. It should be stated that a kind of non-uniqueness phenomena can already be observed for the linear heat equation in $\R^n$ where we need some decay at $\infty$ to ensure uniqueness, so these kinds of issues are not unexpected.
	
	Lastly, does uniqueness also hold among weak solutions? Indeed, this is true, at least for solutions with non-increasing energy (otherwise counterexamples via reverse bubbling exist) and has been shown by Rivi\`ere \cite{riv} and Freire \cite{freire}. Rivi\`ere's argument works for the small energy regime and for the target manifold $S^{n-1}$ and employs an ingenious absorption argument that allows to deduce that the solution is actually a solution in the strong sense. Indeed, the key result may be stated as follows:
	\begin{lem}
	\label{improvedregriv}
		Let $u \in H^{1}(T^2;S^{n-1})$ and $f \in L^{2}(T^2; \R^n)$ and assume that $u$ solves the following non-linear quadratic PDE:
		\begin{equation}
			- \Delta u = u | \nabla u |^2 + f
		\end{equation}
		Then $u \in H^{2}(T^2;S^{n-1})$.
	\end{lem}
	It should be noted that Lemma \ref{improvedregriv} gives the maximal amount of regularity one may expect from a general $u$ solving such an equation. In the context of the harmonic gradient flow, it may be applied to:
	$$- \Delta u(t) = u(t) | \nabla u(t) |^2 - \partial_{t} u(t),$$
	which holds, if $u \in H^{1}(T^2;S^{n-1})$ is a weak solution, for almost all times $t$. Thus, we deduce that for all such times $u(t) \in H^{2}(T^2;S^{n-1})$. However, this does not suffice to prove uniqueness, as we need an $L^{2}$-bound on the $H^{2}$-norms when integrated over time. This follows however immediately by using Ladyzhenskaya's estimate to conclude:
	$$\| \nabla u(t) \|_{L^{4}} \lesssim \| \nabla u(t) \|_{L^2} \| \nabla^2 u(t) \|_{L^2} \lesssim E(u_0) \| u(t) \|_{H^{2}},$$
	so we arrive at, by using elliptic regularity as well as $| u(t) | = 1$ almost everywhere:
	$$\| u(t) \|_{H^{2}} \lesssim \| u(t) \|_{L^2} + \Big{\|} u(t) | \nabla u(t) |^2 - \partial_{t} u(t) \Big{\|}_{L^2} \lesssim 1 + E(u_0) \| u(t) \|_{H^{2}} + \| \partial_{t} u(t) \|_{L^2},$$
	so if $E(u_0)$ is sufficiently small, the $H^2$-norm on the RHS may be absorbed in the LHS and by integrating in time, the desired local integrability follows and allows for the application of Theorem \ref{mainstruwe1} to conclude.
	
	The proof of Lemma \ref{improvedregriv} may be found in \cite{riv} or \cite{wettsteinsphere} for the half-harmonic gradient flow where a natural analogue holds. Indeed, the techniques discussed so far naturally generalise to the framework of the half-harmonic gradient flow. The motivation to study this gradient flow stems once more from approximation of solutions to the half-harmonic map equation as well as the interest in expanding ideas from the local world to the fractional one. This is what the author has done in \cite{wettsteinsphere} for the case of the target manifold being a sphere and in \cite{wettstein} for the target manifold being any closed $N$. Since we shall restrict our considerations later on to the case $N = S^{n-1}$ anyways, let us focus on this special case, where the equation takes the form:
	$$\partial_{t} u + (-\Delta)^{1/2} u = u | d_{1/2} u |^2$$
	The most natural formulation of the half-harmonic gradient flow equation in $N$ is phrased as:
	$$\partial_t u + (-\Delta)^{1/2} u \perp T_u N \quad \text{ in } \mathcal{D}'([0,T[ \times S^{1})$$
	First, one may wonder what was known about the half-harmonic gradient flow before \cite{wettsteinsphere}. In \cite{schisirewang}, the authors had already constructed solutions to the gradient flow of various non-local energies of similar type as the $1/2$-Dirichlet energy by discretisation. Unfortunately, the result was limited to target spaces with inherent symmetry such as the sphere due to the limits taken. In a different paper \cite{sireweizheng}, the bubbling as $t \to + \infty$ was investigated by adapting the inner-outer gluing scheme to the non-local framework, establishing that bubbling is possible for $N = S^{1}$ at time $t = + \infty$. The authors of \cite{sireweizheng} further conjectured that bubbling may actually only occur asymptotically, so no finite time bubbling is possible due to dimensional peculiarities of $\R$ and $S^1$. The conjecture is, according to our best knowledge, still open and under investigation.
	
	Let us now turn to the main result as found in \cite{wettsteinsphere} that answered some of the questions about the half-harmonic gradient flow:
	\begin{thm}
	\label{mainresjw}
		Let $u_0 \in H^{1/2}(S^1;S^{n-1})$ be any initial data. There exists $\varepsilon > 0$, such that if:
		\begin{equation}
			E_{1/2}(u_0) \leq \varepsilon,
		\end{equation}
		then there exists a unique energy class solution $u: \R_{+} \times S^1 \to S^{n-1} \subset \R^n$ of the weak fractional harmonic gradient flow:
		\begin{equation}
		\label{gradflowresjw}
			\partial_t u + (-\Delta)^{1/2} u = u | d_{1/2} u |^2,
		\end{equation}
		satisfying $u(0, \cdot) = u_0$  in the sense $u(t,\cdot) \to u_0$ in $L^2$, as $t \to 0$. Moreover, the solution fulfills the energy decay estimate:
		$$E_{1/2}(u(t)) \leq E_{1/2}(u(s)) \leq E_{1/2}(u_0), \quad \forall t \geq s \in [0, +\infty [.$$
		In fact, $u \in C^{\infty}(]0,\infty[ \times S^1)$ and for an appropriate subsequence $t_k \to \infty$, the sequence $u(t_k)$ converges weakly in $H^{1}(S^1)$ to a point.
		
		If $u_0$ has arbitrary energy, then we still get the existence of a solution to \eqref{gradflowresjw} on some time interval $[0,T[$, where $T = T(u_0) > 0$ depends on the initial datum. There exists a similar characterisation of $T(u_0)$ as in \cite{struwe1}
		\begin{equation}
		\label{concentrationcond}
			\limsup_{t \to T} \varepsilon(R; u, t) \geq \varepsilon_1, \quad \forall R \in ]0,\frac{1}{2}[,
		\end{equation}
		where:
		\begin{equation}
			\varepsilon(R;u,T) := \sup_{x \in S^1, t \in [0,T]} E_{R}(u;x,t) = \sup_{x \in S^1, t \in [0,T]} \frac{1}{2} \int_{B_{R}(x)} | (-\Delta)^{1/4} u(t) |^2 dx,
		\end{equation}
		and $\varepsilon_1 > 0$ is a quantity appearing in the proof of the result and is independent of $u_0, R, T$.
	\end{thm}
	The latter half of Theorem \ref{mainresjw} is implicit in \cite{wettsteinsphere} due to the nature of the local existence proof. As already hinted at earlier, the entire result continues to hold true if we use an arbitrary $N$ instead of $S^{n-1}$. As the proof of this result is similar in spirit, but different in various technical and computational ways, we do not go into some details of the proof here, but we just state that testing \eqref{gradflowresjw} against $u$ and its derivatives again yields the desired control and thus one may deduce existence and regularity similar to \cite{struwe1}, once we have a sufficient local existence theory for smooth $u_0$. For example, again testing against $\partial_t u$ shows:
	\begin{align}
	\int_{0}^{T} \int_{S^1} &| \partial_{t} u |^2 dx dt + E_{1/2}(u(T)) - E_{1/2}(u_0) \notag \\
		&= \int_{0}^{T} \int_{S^1} | \partial_{t} u |^2 dx dt + \int_{0}^{T} \int_{S^1} \partial_{t} \left( \frac{1}{2} | (-\Delta)^{1/4} u |^2 \right) dx dt = 0,
	\end{align}
	again using the fact that $\partial_{t} u \in T_{u} S^{n-1} \perp u$. Local existence can be obtained by linearisation, the Inverse Function Theorem and elliptic regularity considerations as well as a bootstrap technique. Later on in the current work, we shall provide a second proof based on a fixed-point argument that again allows us to show existence for a short period of time. Furthermore, the argument in \cite{riv} generalises thanks to the following Lemma:
	\begin{lem}
	\label{reglemma}
		Let $f \in L^{2}(S^1;\R^n)$ and assume that $u \in H^{1/2}(S^1;S^{n-1})$ solves the following equation:
		\begin{equation}
		\label{eqforimprovedreg}
			(-\Delta)^{1/2} u = u | d_{1/2} u |^{2} + f.
		\end{equation}
		Then, we have the following improved regularity property:
		$$u \in H^{1}(S^1;S^{n-1}).$$
	\end{lem}
	The proof of uniqueness for weak solutions with sufficiently small energy now follows as before in the local case by absorption and employing a fractional Ladyzhenskaya-type estimate that may be proven by means of Fourier coefficients.
	
	For completeness' sake, we would like to mention that in a later work \cite{hydersire}, the authors consider an alternative version of the half-harmonic heat flow not based on the $L^{2}$-gradient flow, but rather on ideas similar to the connection noticed by Caffarelli and Silvestre between fractional Laplacians and Dirichlet-to-Neumann and reading:
	$$(\partial_{t} - \Delta)^{1/2} u \perp T_{u} N,$$
	which allows for a monotonicity formula. Such a formula is unknown in the case of the flow \eqref{gradflowresjw}. It is obvious that both flows allow for the same stationary solutions, namely half-harmonic maps, but the approaches are independent.\\
	
	In our current work, we will use the characterisation \eqref{concentrationcond} above in order to state and prove results regarding bubbling. In fact, we shall show that at most finitely many points exist where bubbles form. By observing that the bubbles are non-constant half-harmonic maps, we may further show that there must localise a quantum of the $1/2$-Dirichlet energy which gets removed from the flow. To summarise, we shall obtain:
	\begin{thm}
	Let $u$ be a solution as in Theorem \ref{mainresjw} and let $x_0 \in S^1$ be a point, such that:
	\begin{equation}
	\label{concentrationinpt}
		\limsup_{t \to T} \int_{B_{R}(x_0)} | (-\Delta)^{1/4} u |^2 dx \geq \varepsilon_1, \quad \forall R > 0,
	\end{equation}
	where $\varepsilon_1 > 0$ is as in \cite[Lemma 4.10]{wettstein}. Then there exists a half-harmonic map $v: \R \to S^{n-1}$, such that:
	\begin{equation}
		u_n \to v \quad \text{ weakly in } H^{1}(\R) \text{ and strongly in } H^{1/2}(\R),
	\end{equation}
	where $u_n$ is a suitable rescaling and translation of $u$.
	\end{thm}
Therefore, by using such considerations, we may conclude that the solution constructed in \cite{wettstein} may be extended by $L^2$-continuity and will be smooth except for finitely many times $0 < T_1 < \ldots < T_n < + \infty$, which may be characterised by concentration identities as in Theorem \ref{mainresjw}:
	\begin{thm}
		Let $u_0 \in H^{1/2}(S^1;S^{n-1})$, then there exists a weak solution of \eqref{gradflowresjw} with non-increasing $1/2$-Dirichlet energy:
		$$u: [0, +\infty [ \times S^1 \to N,$$
		with $u \in L^{\infty}([0,+\infty[; H^{1/2}(S^1;S^{n-1})) \cap H^{1}([0,+\infty[;L^{2}(S^1;S^{n-1}))$ such that, except for finitely many times $0 < T_1 < \ldots < T_{n} < T_{n+1} := + \infty$, the function $u$ is smooth:
		$$u \in C^{\infty}(]T_k, T_{k+1}[; S^{n-1}), \quad \forall k = 1, \ldots n.$$
		Moreover, we may bound the number $n = n(u_0)$ as follows:
		$$n(u_0) \leq \frac{E(u_0)}{\varepsilon_0},$$
		where $\varepsilon_0 > 0$ is the minimum amount of $1/2$-energy a non-constant, half-harmonic map with values in $S^{n-1}$ must possess.
	\end{thm}
	Of course, both of the results for $S^{n-1}$ carry over to the general manifold case $N$ without much difficulty.\\
	
	In a future paper, we will investigate the smoothness at the critical times outside of bubbling points. This issue is quite delicate due to the non-local nature of the equation at hand and thus requires more care than in \cite{struwe1} where suitable localisations are immediately available. Furthermore, the question whether bubbling may even occur remains open and under investigation as well, see also \cite{sireweizheng} for a conjecture in this direction. Additionally, the global existence result is still new in the case of arbitrary target manifolds, as previous papers such as \cite{schisirewang} have either only dealt with special target manifolds or with solutions for a possibly short amount of time as in the author's previous work. Finally, the author is aware of current research by Prof. Michael Struwe concerning the half-harmonic gradient flow based on techniques involving harmonic extensions and, once more, arguments similar to \cite{struwe1}. Thus, alternative approaches to the problem at hand may be possible.\\
	
	The structure of the paper is as follows: In Section \ref{prelim}, we will quickly recall the necessary notions used throughout the paper, in particular fractional gradients and divergences, Triebl-Lizorkin spaces on $S^1$ and the fractional Laplacian. In the following Section \ref{halfharmonic}, we provide proofs and statements for properties of the half-harmonic gradient flow. In particular, in Section \ref{existence}, we provide an alternative proof of local existence of solutions. Then in Section \ref{uniqueness}, we give the details of the proof of Lemma \ref{reglemma}. Afterwards, we investigate bubbling in finite time in Section \ref{bubbling}, studying the concentration of energy. Lastly, Section \ref{global} deals with extensions and other ideas to find global solutions to our main PDE.

\section{Preliminaries} \label{prelim}

In this brief preliminary section, we shall introduce some of the most important notions used throughout. In particular, we discuss Triebel-Lizorkin spaces on $S^1$, provide a short summary of fractional gradient and fractional divergences based on \cite{mazoschi} and finally recall some of the main results associated with the fractional heat flow. Most of the results are discussed in more detail in \cite{wettsteinsphere} and the references provided therein.

	\subsection{Triebel-Lizorkin Spaces on the Unit Circle and Fractional Laplacians}

	Firstly, we shall discuss Triebel-Lizorkin spaces on the unit circle $S^1 \subset \R^2$ and recall some of the most important properties of and formulas for the fractional Laplacian. Much of the current presentation is due to \cite{prats1}, \cite{prats2} and \cite{schmeitrieb}. Throughout, we shall use the distance:
	$$| x-y | = 2 | \sin \left(\frac{x-y}{2} \right),$$
	for all $x,y \in S^1 \simeq \R \mod 2 \pi \Z$.\\
	
	We define for any $f: S^1 \to \R$ the following quantity based on the fractional gradients $d_{s} f(x,y) = \frac{f(x) - f(y)}{| x-y |^s}$:
	\begin{equation}
		\mathcal{D}_{s,q}(f)(x) := \left( \int_{S^1} | d_{s} f(x,y) |^q \frac{dy}{| x-y |} \right)^{1/q},
	\end{equation}
	for all $1 \leq q < \infty$ and $0 < s < 1$. We refer to the next subsection for some details on the fractional gradient $d_s f$. Then:
	\begin{equation}
		\| f \|_{\dot{W}^{s,(p,q)}(S^1)} := \| \mathcal{D}_{s,q}(f)(x) \|_{L^{p}(S^1)},
	\end{equation}
	for every $1 \leq p \leq \infty$. If $p = q$, these spaces correspond to the usual homogeneous Gagliardo-Sobolev spaces $\dot{W}^{s,p}(S^1)$, see \cite{prats1}, \cite{prats2}.\\
	
	Furthermore, we shall denote as per usual by $\mathcal{D}'(S^1)$ the set of all distributions on $S^1$ and occasionally use $\mathcal{D}(S^1)$ as an alternative notation for the space $C^{\infty}(S^1)$. Finally, $\hat{f}(k)$ shall always be the $k$-th Fourier coefficient of $f$, for all $f \in \mathcal{D}'(S^1)$. It is formally defined by:
	\begin{equation}
		\hat{f}(k) := \frac{1}{2\pi} \langle f, e^{-ikx} \rangle = \frac{1}{2\pi} f \left( e^{-ikx} \right), \quad \forall k \in \Z
	\end{equation}
	In \cite{schmeitrieb}, it is shown that one may define Triebel-Lizorkin spaces for $S^1$, denoted by ${F}^{s}_{p,q}(S^1)$, completely analogous to the usual space $F^{s}_{p,q}(\R^n)$ for any parameters $s \in \R$ and $p,q \in [1, \infty[$:
	\begin{equation}
		F^{s}_{p,q}(S^1) := \big{\{} f \in \mathcal{D}'(S^1) \ \big{|}\ \| f \|_{F^{s}_{p,q}} < +\infty \big{\}}
	\end{equation}
	The norm is defined by:
	\begin{equation}
		\| f \|_{F^{s}_{p,q}} := \Bigg{\|} \Bigg{\|} \left( \sum_{k \in \mathbb{Z}} 2^{js} \varphi_{j}(k) \hat{f}(k) e^{ikx} \right)_{j \in \mathbb{N}} \Bigg{\|}_{l^{q}} \Bigg{\|}_{L^{p}(S^1)},
	\end{equation}
	for a suitable partition of unity $(\varphi_{j})_{j \in \mathbb{N}}$ consisting of smooth, compactly supported functions on $\R$ with the properties:
	$$\operatorname{supp} \varphi_{0} \subset B_{2}(0), \quad \operatorname{supp} \varphi_{j} \subset \{ x \in \R\ |\ 2^{j-1} \leq | x | \leq 2^{j+1} \}, \forall j \geq 1$$
	as well as:
	$$\forall k \in \mathbb{N}: \sup_{j \in \mathbb{N}} 2^{jk} \| D^{k} \varphi_j \|_{L^{\infty}} \lesssim 1$$
	One may develop, as for example seen in \cite[Chapter 3]{schmeitrieb}, a complete theory of Triebel-Lizorkin spaces on $S^1$ and more generally on the $n$-torus $T^n$ by following the techniques of these spaces on $\R^n$. We list some of the most important properties in \cite{wettsteinsphere} and refer there for some detailed references in \cite{schmeitrieb}, but for now it suffices to be aware that all tools and results for Triebel-Lizorkin spaces $F^{s}_{p,q}(\R^n)$ are also available for $F^{s}_{p,q}(T^n)$.\\
	
	It turns out that the fractional gradients are exceptionally useful in studying non-local problems. As an example, the following result found in \cite{prats2} is key to many of our arguments, allowing us to restrict our considerations to fractional gradients rather than fractional Laplacians:
	
	\begin{thm}[Theorem 1.2, \cite{prats2}]
	\label{schiwangthm1.4}
		Let $s \in (0,1)$, $p,q \in ]1, \infty[$ and $f \in L^p(\R)$. Then:
		\begin{itemize}
			\item[(i)] We know $\dot{W}^{s, (p,q)}(\R^n) \subset \dot{F}^{s}_{p,q}(\R^n)$ together with:
			\begin{equation}
				\| f \|_{\dot{F}^{s}_{p,q}(\R^n)} \lesssim \| f \|_{\dot{W}^{s, (p,q)}(\R^n)}.
			\end{equation}
			
			\item[(ii)] If $p > \frac{nq}{n + sq}$, then we also have the converse inclusion together with:
			\begin{equation}
			\label{secondpartofthm2.1}
				\| f \|_{\dot{W}^{s, (p,q)}(\R^n)} \lesssim \| f \|_{\dot{F}^{s}_{p,q}(\R^n)}.
			\end{equation}
		\end{itemize}
		The constants depend on $s, p, q, n$.
	\end{thm}
	
	By using the properties in \cite{schmeitrieb} for periodic functions and employing Theorem \ref{schiwangthm1.4}, we can arrive at the following equivalence with Triebel-Lizorkin spaces for all $1 < q < \infty$ and $1 < p < \infty$:
	\begin{equation}
		\dot{W}^{s, (p,q)}(S^1) = \dot{F}^{s}_{p,q}(S^1),
	\end{equation}
	with equivalence of the corresponding seminorms, provided $p > \frac{q}{1+sq}$. As a simple, but important special case, let us observe that if $s = 1/2$ and $q = 2$, then $p > 1$ is the requirement in Theorem \ref{schiwangthm1.4} for the equality of $\dot{F}^{1/2}_{p,2}$ and $\dot{W}^{1/2, (p,2)}$ to hold. Some more details and a proof of one direction of Theorem \ref{schiwangthm1.4} can be found in the appendix of \cite{wettsteinsphere}.\\
	
	Finally, we would like to briefly address the fractional Laplacian. The simplest definition is based on the Fourier multiplier properties of the Laplacian itself, leading ultimately to the following definition for the fractional $s$-Laplacian on Fourier series on $S^1$:
	\begin{equation}
		\widehat{(-\Delta)^{s} f}(k) = | k |^{2s} \hat{f}(k),
	\end{equation}
	for every $k \in \mathbb{Z}$ and all $0 < s < 1$. There is an alternative formulation as a Cauchy principal value, which actually leads to the same operator and is often useful:
	\begin{equation}
		(-\Delta)^{1/2} f(x) = C \cdot P.V. \int_{S^{1}} \frac{f(x) - f(y)}{| x-y |^{2}} dy,
	\end{equation}
	where $C > 0$ denotes a suitable constant. Similar formulas with less explicit kernels exist for $0 < s < 1$, these are omitted for accessibility of the presentation. Additionally, it is of course possible to define the fractional Laplacians also on $\R^n$, leading again to two different characterisations (as a Fourier multiplier and Cauchy principal value) with the same type of formulas. The details are thus omitted.\\
	
	An essential property of function spaces is their behaviour under Fourier multipliers, for example extending results such as Mikhlin's multiplier theorem for $L^p$-spaces. As the fractional Laplacian is obviously a Fourier multiplier operator, one expects characterisations of the spaces $F^{s}_{p,q}(S^1)$ based on these operators, compare with the Bessel potentials. Indeed, one easily sees (\cite{schmeitrieb}):
	$$(-\Delta)^{s}: \dot{F}^{t+2s}_{p,q} \to \dot{F}^{t}_{p,q},$$
	for all $p,q \in (1, \infty)$ and $t, t+2s \in \R$. This should not be surprising and follows along the same lines as in the case of Triebel-Lizorkin spaces on $\R^n$. Observe the use of $\dot{F}^{s}_{p,q}(S^1)$ rather than $F^{s}_{p,q}(S^1)$ indicating the use of homogeneous Triebel-Lizorkin spaces, which are again defined as usual, see also \cite{schmeitrieb}.

	\subsection{Fractional Gradients and Divergences}
	
	Next, we would like to discuss in some depth the notion of fractional gradient and its derivatives, like the fractional divergence and certain weighted $L^p$-spaces. The presentation greatly draws from \cite{mazoschi} and is a shortened version of \cite{wettsteinsphere}:\\
	
	One may introduce $\mathcal{M}_{od}(\R \times \R)$ as the set of all measurable functions $f: \R \times \R \to \R$ with respect to the weighted Lebesgue measure $\frac{dx dy}{ | x-y |}$. In complete analogy, we do the same for $S^1$ instead of $\R$ as the domain, denoting this space by $\mathcal{M}_{od}$ if both $\R$ or $S^1$ are possible as domains. Naturally, the associated $L^p$-spaces, denoted $L^{p}_{od}$ are of interest and the defining (semi-)norms are given by:
	\begin{equation}
		\| F \|_{L^{p}_{od}} := \left( \int \int | F(x,y) |^{p} \frac{dy dx}{| x-y |} \right)^{1/p},
	\end{equation}
	for $1 \leq p < \infty$. The space $L^{\infty}_{od}(S^1 \times S^1)$ and $L^{\infty}_{od}(\R \times \R)$ as the sets of essentially bounded functions with the essential supremum as the (semi-)norm. Later on, the following quantity, defined in terms of $F, G \in \mathcal{M}_{od}$, will be useful:
	\begin{equation}
		F \cdot G (x) := \int F(x,y) G(x,y) \frac{dy}{| x-y |}
	\end{equation}
	In the special case $F = G$, this becomes:
	\begin{equation}
		F \cdot F(x) = | F |^{2}(x), \quad | F |(x) := \sqrt{F \cdot F(x)}
	\end{equation}
	Of course, this shows:
	$$\| F \|_{L^{2}_{od}}^2 = \int | F |^2(x) dx$$
	
	Let us finally turn to the definition of fractional gradients: For a measurable function $f: \R \to \R$ or $f: S^1 \to \R$, we define for $0 \leq s < 1$ the fractional $s$-gradient by:
	$$d_s f(x,y) = \frac{f(x) - f(y)}{| x-y |^s} \in \mathcal{M}_{od},$$
	and the corresponding $s$-divergence by means of duality, i.e. for $F \in \mathcal{M}_{od}$:
	\begin{equation}
		\langle \div_{s} F, \varphi \rangle = \int \int F(x,y) d_{s} \varphi(x,y) \frac{dy dx}{| x-y |}, \quad \forall \varphi \in C^{\infty}_{c}(\R) \text{ or } C^{\infty}(S^1)
	\end{equation}
	It is obvious that:
	$$d_{s} f(y,x) = - d_{s} f(x,y)$$
	Also, a version of Leibniz' rule holds true:
	$$d_{s} \left( fg \right)(x,y) = d_{s} f(x,y) g(x) + f(y) d_{s} g(x,y)$$
	Naturally, $\div_{s} F$ is only well-defined in a distributional sense.\\
	
	Using the notions introduced for functions $F \in L^{p}_{od}$ and as we have already stated in the subsection before, we do now have:
	\begin{equation}
		\| | d_{s} f | \|_{L^{p}(S^1)} = \| f \|_{\dot{W}^{s,(p,2)}(S^1)},
	\end{equation}
	We refer to Theorem \ref{schiwangthm1.4} for the significance of this. Finally, the fractional Laplacian also has a place in the setting of fractional gradients and divergences, behaving much as expected from $\Delta = \div \circ \nabla$:
	\begin{equation}
	\label{fraclaplbygrad}
		(- \Delta)^{s} f =  C_s\div_{s} d_s f,
	\end{equation}
	for some constant $C_s > 0$ depending on $s$. Equation \eqref{fraclaplbygrad} has to be read as follows:
	$$C_s \int d_s f \cdot d_s g (x) dx = \int (-\Delta)^{s} f \cdot g dx = \int (-\Delta)^{s/2} f \cdot (-\Delta)^{s/2} g dx$$
	A key result to establish, for instance, regularity of fractional harmonic maps or the uniqueness of weak solutions to the half-harmonic gradient flow with small initial energy is the following Wente-type estimate (which was already mentioned in the introduction, but included once more for ease of presentation):
	
	\begin{lem}[Theorem 2.1, \cite{mazoschi}]
	\label{fractionalwentelemma}
		Let $s \in (0,1)$ and $p \in (1,\infty)$. For $F \in L^{p}_{od}(\R \times \R)$ and $g \in \dot{W}^{s,p'}(\R)$, where $p'$ denotes the H\"older dual of $p$, we assume that $\div_s F = 0$. Then $F \cdot d_s g$ lies in the Hardy space $\mathcal{H}^{1}(\R)$\footnote{The Hardy space $\mathcal{H}^{1}(\R)$ is the subspace of $L^{1}(\R)$-functions such that: $$M_{\Phi}(f)(x) := \sup_{t > 0} | \Phi_{t} \ast f |(x) \in L^{1}(\R),$$ where $\Phi$ is a Schwartz function on $\R$ with $\int \Phi dx = 1$ and $\Phi_{t}(x) = 1/t \cdot  \Phi(x/t)$. Various other, sometimes simpler characterisations (for example $\mathcal{H}^{1}(\R) \simeq F^{0}_{1,2}(\R)$) exist and the relevance of Hardy spaces stem from their "good" behaviour with respect to Cald\'eron-Zygmund operators, especially when compared to $L^{1}(\R)$.} and we have the estimate:
		\begin{equation}
			\| F \cdot d_s g \|_{\mathcal{H}^{1}(\R)} \lesssim \| F \|_{L^{p}_{od}(\R \times \R)} \cdot \| g \|_{\dot{W}^{s,p'}(\R)}.
		\end{equation}
	\end{lem}
	
	In the case where $s = 1/2$ and $p = p' = 2$, we may immediately deduce $F \cdot d_s g \in H^{-1/2}(\R)$ following the Sobolev embedding $\dot{H}^{1/2}(\R) \hookrightarrow BMO(\R)$ with analogous estimates. Naturally, the result also remains valid in the case of the domain being $S^1$:
	
		\begin{lem}
	\label{fractionalwentelemmaons1insec2}
		For $F \in L^{2}_{od}(S^1 \times S^1)$ and $g \in \dot{H}^{1/2}(S^1)$, we assume that $\div_{1/2} F = 0$. Then $F \cdot d_{1/2} g$ lies in the space $H^{-1/2}(S^1)$ and we have the estimate:
		\begin{equation}
			\| F \cdot d_{1/2} g \|_{{H}^{-1/2}(S^1)} \lesssim \| F \|_{L^{2}_{od}(S^1 \times S^1)} \cdot \| g \|_{\dot{H}^{1/2}(S^1)}.
		\end{equation}
	\end{lem}
	
	We refer to \cite{wettsteinsphere} for some details of the proof.
	
	\subsection{Fractional Heat Flow}
	
	The presentation of this subsection follows \cite{garofalo} and we refer to it and the references mentioned therein for details.\\
	
	A natural problem to consider the fractional heat flow $\partial_t + (-\Delta)^{s}$. One may be motivated by the ubiquity of the heat equation in general mathematics or by the interest in the fractional harmonic gradient flow, whose linearisation is closely connected to this equation. Of course, semi-group theory provides a suitable theoretical framework to discuss questions of existence, regularity and uniqueness of such solutions. For our purpose, it will be sufficient to introduce the heat kernel (at least in the special case $s = 1/2$) and discuss some of its basic properties.
	
	A natural approach to solve the equation for the fundamental solution of the homogeneous equation:
	$$\partial_t u + (-\Delta)^{s} u = 0, \quad u(0,x) = \delta_{0}(x),$$
	on $[0,\infty[ \times \R$ would be to apply a spatial Fourier transform, leading to the following equation for the Fourier transform:
	$$\partial_{t} \hat{u}(t,\xi) + | \xi |^{2s} \hat{u}(t,\xi) = 0$$
	Solving this ODE for fixed $\xi$ leads us to:
	\begin{equation}
		\hat{u}(t, \xi) = e^{-| \xi |^{2s} t} \cdot \hat{\delta_{0}}(\xi) = e^{-| \xi |^{2s} t}
	\end{equation}
	The fundamental solution is thus the Fourier inverse of this expression and in the case $s = 1/2$, the following explicit formula exists:
	\begin{equation}
		u(t,x) = C \cdot \frac{t}{t^2 + x^2},
	\end{equation}
	$C$ being a suitable constant. A fundamental solution on $S^1$ may be constructed by periodic extension, so we discover an analogous kind of heat kernel. It should be noted that outside of $(t,x) = (0,0)$, the heat kernel is smooth, thus implying the smoothing property already well-known from the standard heat flow.
	
	The fractional heat semigroup may be used for various things, such as a formula for the fractional Laplacians by subordination, see \cite{garofalo}. We are more interested in the immediate regularity properties. By using Duhamel's principle, one may indeed solve the problem:
	\begin{align}
		\partial_{t} v(t,x) + (-\Delta)^{s} v(t,x) 	&= f(t,x), \quad &\forall (t,x) \in ]0,\infty[ \times S^1 \\
		v(0,x) 	&= g(x), \quad &\forall x \in S^1
	\end{align}
	Regularity may be obtained either by semigroup theory or, if $s = 1/2$, using the ellipticity of $\partial_{t} + (-\Delta)^{1/2}$ which is contained in:
	$$\left( \partial_{t} - (-\Delta)^{1/2} \right) \left( \partial_{t} + (-\Delta)^{1/2} \right) = \left( \partial_{t} + (-\Delta)^{1/2} \right) \left( \partial_{t} - (-\Delta)^{1/2} \right) = \partial_{t}^2 + \partial_{x}^2$$
	Therefore, if $s = 1/2$, an $L^p$-theory with estimates as expected does exist, see also \cite{hieber} and the discussion in \cite{wettsteinsphere} on the regularity of local solutions.

\section{Half-Harmonic Gradient Flow} \label{halfharmonic}

In this subsection, we go into some depth regarding some specific aspects of the proof of Theorem \ref{mainresjw}. To be precise, we shall supply the reader with an alternative approach to the local existence result for smooth $u_0$ for some brief interval of time based on Banach's fixed point theorem, present a detailed account of the proof of Lemma \ref{reglemma}, since this argument is beautiful and provides potential insight into the way Hodge decomposition may be substituted in the non-local case. The way to conclude from this uniqueness for weak solutions follows by using similar arguments as in the introduction and we refer to \cite{wettsteinsphere} for the details. Following this, we shall then discuss bubbling processes based on concentration estimates in localised Gagliardo seminorms and rescaled versions of the solution. The approach is quite similar to \cite{struwe1}, however the non-locality renders quite a few steps more difficult and requires us to refine an estimate we have previously proven in \cite{wettsteinsphere} and \cite{wettstein}. Only after having this estimate available are we in a position to address boundedness of suitable rescalings of the solution to the half-harmonic gradient flow. To conclude this section, we discuss global existence of solutions by using two distinct approaches, one producing a solution based on Theorem \ref{mainresjw} with non-increasing energy, while the other proves existence based on variational arguments, but does not immediately exhibit monotonicity of energy.

\subsection{A Local Existence and Regularity Result} \label{existence}

In this first subsection, our goal is to prove the following:

\begin{prop}
\label{localexistence}
	Let $u_{0} \in C^{\infty}(S^1;S^{n-1})$. Then there exists a solution $u: [0,T] \times S^{1} \to S^{n-1}$ with $u(0,\cdot) = u_0$ of the equation \eqref{gradflowresjw} which is smooth on some time interval $[0,T]$, where $T = T(u_0)$.
\end{prop}

This result was already proven in \cite{wettsteinsphere} by introducing an appropriate solution operator $H$ and applying the inverse function theorem. The key observations were that firstly, the linearisation of $H$ at any smooth function is indeed Fredholm with index 0 and thus injectivity and invertibility become equivalent. Secondly, an argument based on maximum principles shows that smooth elements in the kernel of the linearisation are always trivial. Bootstrapping to deduce sufficient regularity then bridges the gap between the two observations and amounts to the existence result stated as Proposition \ref{localexistence}.\\

Here, we will take a slightly different approach and substitute the use of Fredholm theory by employing a standard argument based on Banach's fixed point theorem. For the remainder of this section, we shall denote by:
$$W^{1,p}_{u_0}([0,T] \times S^{1}) := \big{\{} u \in W^{1,p}([0,T] \times S^{1}) \in \big{|} u(0, \cdot) = u_0 \big{\}},$$
where $u_0 \in C^{\infty}(S^1;S^{n-1})$ is a given boundary value.
Indeed, we shall prove:

\begin{lem}
\label{fixedpoint}
	Let $u_0 \in C^{\infty}(S^{1};S^{n-1})$ and $p > 4$. Then the map:
	\begin{equation}
	\label{contop}
		S: W^{1,p}_{u_0}([0,T] \times S^{1}) \to W^{1,p}_{u_0}([0,T] \times S^{1}),
	\end{equation}
	mapping $u \in W^{1,p}_{u_0}([0,T] \times S^{1})$ to the unique solution $T(u) \in W^{1,p}_{u_0}([0,T] \times S^{1})$ of the following system:
	\begin{align}
		\partial_{t} S(u) + (-\Delta)^{1/2} S(u) 	&= u | d_{1/2} u |^2, 	&\quad \forall (t,x) \in [0,T] \times S^{1} \\
		S(u)(0, x) 	&= u_0(x),		&\quad \forall x \in S^1
	\end{align}
	Given $R > 0$ sufficiently big and $T > 0$ sufficients small, then $S$ is a contraction of the closed ball of radius $R$ around $u_0$, denoted $B_{R}(u_0)$, onto itself and thus possesses a fixed point.
\end{lem}

We remark at this point that by then employing the same kind of bootstrap procedure as in \cite{wettstein}, we immediately deduce that the fixed point is smooth, thus Proposition \ref{localexistence} holds, once we have established Lemma \ref{fixedpoint}. The reader should notice that we tacitly omit any assumption ensuring $u(t,x) \in S^{n-1}$ for $(t,x) \in [0,T] \times S^{1}$. This is no oversight, but relates to the fact that by employing the maximum principle for parabolic equations immediately proves this from the equation \eqref{gradflowresjw}, see \cite{wettsteinsphere}.

\begin{proof}
	First, one should observe that $u | d_{1/2} u |^2 \in L^{\infty}([0,T] \times S^1) \subset L^{p}([0,T] \times S^1)$. This follows by Sobolev embeddings into H\"older spaces and the compactness of the domain. Therefore, the operator $S$ is actually well-defined.\\
	
	By abuse of notation, we denote by $u_0$ also its extension to $[0,T] \times S^1$ which is independent of time. Let us consider the following for arbitrary $u,v \in W^{1,p}_{u_0}([0,T] \times S^{1})$:
	\begin{align}
	\label{lipschitz}
		\| S(u) - S(v) \|_{W^{1,p}}		&\lesssim \| u | d_{1/2} u |^2 - v | d_{1/2} v |^2 \|_{L^{p}} \notag \\
								&\lesssim \| u - v \|_{L^{p}} \| u \|_{C^{1/2}}^2 + \| v \|_{L^{\infty}} \| | d_{1/2} (u-v) | \|_{L^{p}} \left( \| u \|_{C^{1/2}} + \| v \|_{C^{1/2}} \right) \notag \\
								&\lesssim T^{1/p} \| u - v \|_{L^{\infty}} \| u \|_{W^{1,p}} + \| v \|_{W^{1,p}} \cdot T^{1/p} \| u - v \|_{C^{1/2}}  \cdot \left( \| u \|_{W^{1,p}} + \| v \|_{W^{1,p}} \right) \notag \\
								&\lesssim \left( R + \| u_0 \|_{W^{1,p}} \right)^2 \cdot T^{1/p} \cdot \| u - v \|_{W^{1,p}},
	\end{align}
	where we emphasise that all estimates have no further dependence on $T$. This may be seen by mirror-extensions and applying the Sobolev embeddings on potentially larger sets. Thus, we may conclude, provided $R$ is given:
	$$S \text{ is a contraction, if } T \text{ is sufficently small.}$$
	Thus, it remains to be seen that provided $R > 0$ is sufficiently large, then for every $u \in B_{R}(u_0)$, we also have:
	$$S(u) \in B_{R}(u_0)$$
	To see this, we have to consider the difference:
	$$d := \| u_0 - S(u_0) \|_{W^{1,p}}$$
	We define for now $R = 2d$ and then choose $T>0$ so small, that the Lipschitz constant in \eqref{lipschitz} is $1/2$. Let us notice that for any $u \in B_{R}(0)$, we have:
	\begin{align}
		\| u_0 - S(u) \|_{W^{1,p}}	&\leq \| u_0 - S(u_0) \|_{W^{1,p}} + \| S(u_0) - S(u) \|_{W^{1,p}} \notag \\
							&\leq \frac{R}{2} + \frac{1}{2} \| u_0 - u \|_{W^{1,p}} \notag \\
							&\leq \frac{R}{2} + \frac{R}{2} = R,
	\end{align}
	thus:
	$$S(u) \in B_{R}(u_0).$$
	This now concludes the proof of Lemma \ref{fixedpoint}, as $W^{1,p}_{u_0}([0,T] \times S^1)$ is a complete metric space due to the continuity of the trace operator.
\end{proof}

\subsection{Uniqueness: Rivi\`ere's Lemma \ref{reglemma}} \label{uniqueness}

In this short section, we shall explain the proof of Lemma \ref{reglemma}. Recall that we are interested in solutions $u \in H^{1/2}(S^1;S^{n-1})$ of an equation of the form:
\begin{equation}
\label{riviereformrepeat}
	(-\Delta)^{1/2} u = u | d_{1/2} u |^2 + f,
\end{equation}
where $f \in L^{2}(S^1;\R^n)$. We observe the following (using Einstein's summation convention):
\begin{align}
	&\langle \div_{1/2} \left( u_i d_{1/2} u_j - u_j d_{1/2} u_i \right), \varphi \rangle_{\mathcal{D}'(S^1)} \notag \\
	&= \langle u_i d_{1/2} u_j - u_j d_{1/2} u_i, d_{1/2} \varphi \rangle_{L^{2}_{od}(S^1 \times S^1)} \notag \\
	&= \int_{S^1} \int_{S^1} u_i(x) \frac{(u_j(x) - u_j(y))(\varphi(x) - \varphi(y)}{| x-y |^2} - u_j(x) \frac{(u_i(x) - u_i(y))(\varphi(x) - \varphi(y)}{| x-y |^2} dy dx \notag \\
	&= \int_{S^1} \int_{S^1} d_{1/2} u_j(x,y) \left( d_{1/2} (u_i \varphi)(x,y) - d_{1/2} u_i(x,y) \varphi(y) \right) \frac{dy dx}{| x-y |} \notag \\
	&- \int_{S^1} \int_{S^1} d_{1/2} u_i(x,y) \left( d_{1/2} (u_j \varphi)(x,y) - d_{1/2} u_j(x,y) \varphi(y) \right) \frac{dy dx}{| x-y |} \notag \\
	&= \int_{S^{1}} \int_{S^1} d_{1/2} u_i(x,y) d_{1/2} (u_j \varphi)(x,y) - d_{1/2} u_j(x,y) d_{1/2} (u_i \varphi)(x,y) \frac{dy dx}{| x-y |} \notag \\
	&= \int_{S^1} u_i(x) | d_{1/2} u |^2(x) \cdot u_j(x) \varphi(x) - u_j(x) | d_{1/2} u |^2(x) \cdot u_{i}(x) \varphi(x) - f_i(x) u_j(x) \varphi(x) + f_{j}(x) u_{i}(x) \varphi(x) dx \notag \\
	&= \int_{S^1} \left( u_{i}(x) f_{j}(x) - u_{j}(x) f_{i}(x) \right) \varphi(x) dx,
\end{align}
which reveals, in analogy to \cite{riv}:
\begin{equation}
\label{almostconserv}
	\forall i,j \in \{ 1, \ldots, n \}: \div_{1/2} \left( u_i d_{1/2} u_j - u_j d_{1/2} u_i \right) = u_i f_j - u_j f_i
\end{equation}
One may solve now for $i, j$ as above the equation:
$$(-\Delta)^{1/2} \psi_{ij} = u_i f_j - u_j f_i,$$
for $\psi_{ij} \in H^{1}(S^1)$. Observe that we may choose these in such a way that $\psi_{ij} = - \psi_{ji}$. Then it becomes clear:
$$\div_{1/2} \left( u_i d_{1/2} u_j - u_j d_{1/2} u_i - d_{1/2} \psi_{ij} \right) = 0$$
Thus, defining $\Omega_{ij} := u_i d_{1/2} u_j - u_j d_{1/2} u_i - d_{1/2} \psi_{ij}$, we find:
\begin{equation}
\label{rewriting1}
	(-\Delta)^{1/2} u = \Omega \cdot d_{1/2} u + T(u) + d_{1/2} \psi \cdot d_{1/2} u + f
\end{equation}
Here, $T(u)$ is the remainder as already found in \cite{mazoschi} and \cite{wettsteinsphere}: It is given by $T(u) = (T^{1}(u), \ldots, T^{n}(u))$ and
$$\forall i \in \{ 1, \ldots, n \}: T^{i}(u) := \frac{1}{2} \sum_{k=1}^{n} \int_{S^{1}} d_{1/2} u_{i}(x,y) | d_{1/4} u_{k} (x,y) |^2 \frac{dy}{| x-y |}$$
One may generalise this remainder as follows:
$$T^{i}(u,v,w) := \frac{1}{2} \sum_{k=1}^{n} \int_{S^{1}} d_{1/2} u_{i}(x,y)  d_{1/4} v_{k} (x,y) d_{1/4} w_{k}(x,y) \frac{dy}{| x-y |},$$
such that $T(u) = T(u,u,u)$. This term has good integrability properties, see the previous section. To simplify, let us notice that $\psi \in H^{1}(S^1) \hookrightarrow W^{1/2, p}(S^1)$ by Sobolev embeddings for every $p < + \infty$ and thus, using \cite{prats1}:
$$d_{1/2} \psi \cdot d_{1/2} u \in L^{q}(S^1), \quad \forall 1 \leq q < 2,$$
since $| d_{1/2} u | \in L^{2}(S^1)$ by $u \in H^{1/2}(S^1)$. Thus, \eqref{rewriting1} can be rephrased as:
\begin{equation}
\label{rewriting2}
	(-\Delta)^{1/2} u = \Omega \cdot d_{1/2} u + T(u,u,u) + \tilde{f},
\end{equation}
where $\tilde{f} := f + d_{1/2} \psi \cdot d_{1/2} u \in L^{q}(S^1)$ for all $1 \leq q < 2$.\\

The key idea in \cite{riv} is now the following: We try to approximate $\Omega$ by a smooth $\tilde{\Omega}$ with vanishing $1/2$-divergence, such that:
$$\| \Omega - \tilde{\Omega} \|_{L^{2}_{od}} \leq \varepsilon,$$
for $\varepsilon > 0$ small. Similarily, we approximate $u$ by a smooth $\tilde{u}$ in $H^{1/2}(S^1)$. Then \eqref{rewriting2} leads us to:
\begin{equation}
\label{rewriting3}
	(-\Delta)^{1/2} u - \left( \Omega - \tilde{\Omega} \right) \cdot d_{1/2} - T(u,u - \tilde{u},u - \tilde{u}) = \tilde{\Omega} \cdot d_{1/2} u + T(u,u, \tilde{u}) + T(u, \tilde{u}, u- \tilde{u}) + \tilde{f} =: \hat{f}
\end{equation}
Since $\div_{1/2} (\Omega - \tilde{\Omega}) = 0$, we notice that we are in the realm of the fractional Wente-type estimate in Proposition \ref{fracwente}. Namely, if $v \in \dot{F}^{1/2}_{p,2}(S^1)$ for some $p > 2$, then we have by H\"older's inequality:
$$\Big{\|} \left( \Omega - \tilde{\Omega} \right) \cdot d_{1/2} v \Big{\|}_{L^{\frac{2p}{p+2}}} \lesssim \| \Omega - \tilde{\Omega} \|_{L^{2}_{od}} \| v \|_{\dot{F}^{1/2}_{p,2}} \leq \varepsilon \cdot \| v \|_{\dot{F}^{1/2}_{p,2}}$$
Since $F^{1/2}_{p',2}(S^1) \hookrightarrow L^{\frac{2p'}{2-p'}}$ for $p'$ the H\"older conjugate of $p$ by Sobolev embedding, the inequality above immediately yields:
$$\Big{\|} \left( \Omega - \tilde{\Omega} \right) \cdot d_{1/2} v \Big{\|}_{F^{-1/2}_{p,2}} \lesssim \varepsilon \cdot \| v \|_{\dot{F}^{1/2}_{p,2}}$$
In the case $p = 2$, i.e. $v \in \dot{F}^{1/2}_{2,2}(S^1) = \dot{H}^{1/2}(S^1)$, then by Proposition \ref{fracwente} we get immediately:
$$\Big{\|} \left( \Omega - \tilde{\Omega} \right) \cdot d_{1/2} u \Big{\|} \lesssim \| \Omega - \tilde{\Omega} \|_{L^{2}_{od}} \| v \|_{\dot{H}^{1/2}} \leq \varepsilon \cdot \| v \|_{\dot{H}^{1/2}}$$
In an analogous manner, the estimates in the preliminary section show us:
$$\| T(v, u - \tilde{u}, u - \tilde{u}) \|_{\dot{F}^{-1/2}_{p,2}} \lesssim \| u - \tilde{u} \|_{H^{1/2}}^2 \| v \|_{\dot{F}^{1/2}_{p,2}} \leq \varepsilon^2 \cdot \| v \|_{\dot{F}^{1/2}_{p,2}},$$
for all $p \geq 2$. This shows us that the operator:
$$\tau: v \mapsto v - (-\Delta)^{-1/2} \left( \left( \Omega - \tilde{\Omega} \right) \cdot d_{1/2} v + T(v, u - \tilde{u}, u - \tilde{u}) \right),$$
actually defines an invertible operator (one has to be slightly careful at this point and restrict to $v$ having vanishing mean), for any $p \geq 2$, from $\dot{F}^{1/2}_{p,2}(S^1)$ to itself, provided $\varepsilon > 0$ is sufficiently small. \\

Keeping the RHS of \eqref{rewriting3} in mind, it is immediate that it lies in $L^{q}$ for all $1 \leq q < 2$ by estimates from the previous subsection. Thus:
$$(-\Delta)^{-1/2} \left(  \tilde{\Omega} \cdot d_{1/2} u + T(u,u, \tilde{u}) + T(u, \tilde{u}, u- \tilde{u}) + \tilde{f} \right) \in \dot{F}^{1}_{q,2}(S^1) \hookrightarrow \dot{F}^{1/2}_{\frac{2q}{2-q},2}(S^1),$$
for again all $q \in [1,2[$.\\

The conclusion of Lemma \ref{reglemma} follows now by noticing that $\tau(v) = (-\Delta)^{-1/2} \hat{f}$ does possess a solution $v \in \dot{F}^{1/2}_{p,2}(S^1)$ by invertibility for each fixed $p \geq 2$, provided $\varepsilon > 0$ is sufficiently small. Observing that due to compactness of $S^1$, we have:
$$\dot{F}^{1/2}_{p,2}(S^1) \subset \dot{F}^{1/2}_{2,2}(S^1) = H^{1/2}(S^1),$$
by choosing $\varepsilon > 0$ so small, that invertibility holds for $p = 2$ and some $p > 2$, we conclude that the solution $v \in \dot{F}^{1/2}_{p,2}(S^1)$ must also lie in $H^{1/2}(S^1)$. Since $u \in H^{1/2}(S^1)$ is already a solution and by invertibility actually the unique one, we deduce:
$$v = u \Rightarrow u \in \dot{F}^{1/2}_{p,2}(S^1)$$
As $p > 2$ was arbitrary up to possibly choosing better approximations for even smaller $\varepsilon > 0$, we find:
$$u \in \dot{F}^{1/2}_{p,2}(S^1), \quad \forall p \in [1,+\infty[$$
Taking $p = 4$, we deduce:
\begin{equation}
\label{l4int}
	| d_{1/2} u |^2 \in L^{2}(S^1),
\end{equation}
which combined with $| u | = 1$ almost everywhere and \eqref{riviereformrepeat}, we thus conclude:
$$(-\Delta)^{1/2} u = u | d_{1/2} u |^2 + f \in L^2(S^1) \Rightarrow u \in H^{1/2}(S^1;S^{n-1}),$$
which is the required conclusion.\\

All that remains to do is to justify the approximation of $\Omega$ and $u$. Since the latter is standard and does not require any further interesting considerations, it is omitted here. The former, however, requires some care. Thus, let $\varepsilon > 0$ be arbitrary and we shall consider the following approximation:
$$\Omega_{\delta} := \Omega \cdot 1_{D_{\delta}},$$
where:
$$\forall \delta > 0: D_{\delta} := \{ (x,y) \in S^1 \times S^1 | | x-y | \geq \delta \}$$
So $D_\delta$ omits a neighbourhood of the diagonal. It is clear by Lebesgue's dominated convergence, that:
\begin{equation}
	\Omega_{\delta} \to \Omega, \text{ as } \delta \to 0,
\end{equation}
in the space $L^{2}_{od}(S^1 \times S^1)$. Thus, take $\delta$ so small that:
\begin{equation}
	\| \Omega - \Omega_{\delta} \|_{L^{2}_{od}} < \frac{\varepsilon}{2}
\end{equation}
Now, we may argue by convolution by a suitable smooth kernel to replace $\Omega_{\delta}$ by a smooth function, again denoted $\Omega_{\delta}$, which vanishes close to the diagonal $\{ x = y \}$. This is again standard and thus omitted.\\

The final obstacle to overcome is to adapt $\Omega_{\delta} \in C^{\infty}(S^1 \times S^1)$ in such a way that:
$$\div_{1/2} \Omega_{\delta} = 0$$
To achieve this, we shall solve the following problem:
$$(-\Delta)^{1/2} h_{\delta} = \div_{1/2} \Omega_{\delta},$$
i.e. solving the weak equation:
$$\langle (-\Delta)^{1/2} h_{\delta}, \varphi \rangle = \int_{S^1} \int_{S^1} d_{1/2} h_{\delta}(x,y) d_{1/2} \varphi(x,y) \frac{dy dx}{| x-y |} = \int_{S^1} \int_{S^1} \Omega_{\delta}(x,y) d_{1/2} \varphi(x,y) \frac{dy dx}{| x-y |}, \quad \forall \varphi \in C^{\infty}(S^1).$$
Existence of such a solution is immediate, as in the case of $\Omega_{\delta}$, one may define the divergence directly as a smooth function. One may immediately notice that since $\Omega_{\delta} \in L^{2}_{od}$, we have:
$$(-\Delta)^{1/2} h_{\delta} \in H^{-1/2}(S^1) \Rightarrow h_{\delta} \in H^{1/2}(S^1),$$
together with the estimate:
\begin{align}
	\| h_{\delta} \|_{\dot{H}^{1/2}}^2	&= \int_{S^1} | d_{1/2} h_{\delta} |^2 dx \notag \\
							&= \int_{S^1} \left( \Omega_{\delta}(x,y) - \Omega(x,y) \right) d_{1/2} h_{\delta}(x,y) \frac{dy dx}{| x-y |} \notag \\
							&\lesssim \| \Omega_{\delta} - \Omega \|_{L^{2}_{od}} \| h_{\delta} \|_{\dot{H}^{1/2}},
\end{align}
ultimately proving:
$$\| h_{\delta} \|_{\dot{H}^{1/2}} \lesssim \| \Omega_{\delta} - \Omega \|_{L^{2}_{od}} \leq \frac{\varepsilon}{2},$$
where we used in the computation above that $\div_{1/2} \Omega = 0$. Therefore, by choosing $\delta > 0$ sufficiently small, we have:
\begin{equation}
	\| \Omega_{\delta} - d_{1/2} h_{\delta} - \Omega \|_{L^{2}_{od}} \lesssim \varepsilon,
\end{equation}
as well as:
\begin{equation}
	\div_{1/2} \left( \Omega_{\delta} - d_{1/2} h_{\delta} \right) = \div_{1/2} \Omega_{\delta} - (-\Delta)^{1/2} h_{\delta} = 0.
\end{equation}
It should be emphasised that $(-\Delta)^{1/2} = \div_{1/2} \circ d_{1/2}$ in complete analogy to $\Delta = \div \circ \nabla$. This is precisely the desired approximation and thus concludes the proof of Lemma \ref{reglemma}. $\qed$

\subsection{Bubbling-Analysis and Concentration of Energy} \label{bubbling}

In the remaining two subsections, we will treat two new results: General global existence (extending, for example \cite{schisirewang}) and bubbling in finite time (thus investigating the behaviour in critical times more closely). Both questions have been addressed in the local framework in the case of the harmonic gradient flow, see \cite{struwe1}, however similar considerations in the non-local world require some care in adapting the arguments. For example, the rescaling technique is not immediately applicable on $S^1$ and the non-locality of the equation necessitates an investigation of the limiting equation in detail. As a result, we shall present the proofs in detail and provide insight into the mechanisms behind the bubbling and global existence theorem.\\

In the current section, we will first study the concentration of energy in greater detail and with more precise estimates. Two main results shall be obtained: Firstly, we shall improve the following Lemma 3.16 in \cite{wettsteinsphere}:

\begin{lem}
	\label{lemmajw}
	There exist $C >  0$ not depending on $R, u, T$, such that for any smooth $u$ on $[0,T] \times S^1$ and $0 < R < 1$, the following estimate holds for all $x_0 \in S^1$:
	\begin{align}
		\int_{0}^{T} \int_{B_{\frac{3R}{4}}(x_0)} | (-\Delta)^{1/4} u |^4 dx dt 	&\leq C \sup_{0\leq t \leq T} \int_{B_{R}(x_0)} | (-\Delta)^{1/4} u(t) |^2 dx \notag \\
															&\cdot \left( \int_{0}^{T} \int_{B_{R}(x_0)} | (-\Delta)^{1/2} u |^2 dx dt + \frac{1}{R^2} \int_{0}^{T} \int_{S^1} | (-\Delta)^{1/4}u |^2 dx dt \right),
	\end{align}
	by density the same result applies for all $u \in H^{1}([0,T] \times S^1)$, and all boundary terms $u_0 = u(0, \cdot) \in H^{1/2}(S^1)$, with bounded $1/2$-Dirichlet energy. Similarily, we have:
	\begin{align}
		\label{struweest02old}
		\int_{0}^{T} \int_{S^1} | (-\Delta)^{1/4} u |^4 dx dt 	&\lesssim \sup_{0\leq t \leq T, x \in S^1} \int_{B_{R}(x)} | (-\Delta)^{1/4} u(t) |^2 dx \notag \\																			&\cdot \left( \int_{0}^{T} \int_{S^1} | (-\Delta)^{1/2} u |^2 dx dt + \frac{1}{R^3} \int_{0}^{T} \int_{S^1} | (-\Delta)^{1/4}u |^2 dx dt \right).
	\end{align}
\end{lem}

The improvement will be in the order of power of $R$ that occurs and this is indeed crucial for a non-local rescaling argument to work. Namely, we shall show that $R^{-3}$ may be replaced by $R^{-2}$ which allows for suitable rescaling and a blow-up procedure. Secondly, we will connect the condition \eqref{concentrationcond} to an analogous condition for the localised energy in balls, sacrificing potentially focus by allowing for "larger" balls in which the localised Gagliardo-seminorms are bounded from below. Observe that due to the non-local nature of the $1/4$-Laplacian, $\varepsilon(R; u, t)$ takes into account not only value of $u(t,x)$ in a ball, but on the entire $S^1$. However, contributions "far away" are not as important (these are dealt with by enlarging the balls under consideration) and thus we may restrict our attention to the local Gagliardo seminorms on balls.

\subsubsection{An Improved Version of Lemma 3.16 in \cite{wettsteinsphere}}

In this brief subsection, we shall argue why the following refinement of Lemma 3.16 in \cite{wettsteinsphere} holds true:

\begin{lem}
	\label{improvedlemmajw}
	There exist $C >  0$ not depending on $R, u, T$, such that for any smooth $u$ on $[0,T] \times S^1$ and $0 < R < 1/2$, the following estimate holds for all $x_0 \in S^1$:
	\begin{align}
	\label{struweest01}
		\int_{0}^{T} \int_{B_{\frac{3R}{4}}(x_0)} | (-\Delta)^{1/4} u |^4 dx dt 	&\leq C \sup_{0\leq t \leq T} \int_{B_{R}(x_0)} | (-\Delta)^{1/4} u(t) |^2 dx \notag \\
															&\cdot \left( \int_{0}^{T} \int_{B_{R}(x_0)} | (-\Delta)^{1/2} u |^2 dx dt + \frac{1}{R^2} \int_{0}^{T} \int_{S^1} | (-\Delta)^{1/4}u |^2 dx dt \right),
	\end{align}
	by density the same result applies for all $u \in H^{1}([0,T] \times S^1)$, and all boundary terms $u_0 = u(0, \cdot) \in H^{1/2}(S^1)$, with bounded $1/2$-Dirichlet energy. Similarily, we have:
	\begin{align}
		\label{struweest02}
		\int_{0}^{T} \int_{S^1} | (-\Delta)^{1/4} u |^4 dx dt 	&\lesssim \sup_{0\leq t \leq T, x \in S^1} \int_{B_{R}(x)} | (-\Delta)^{1/4} u(t) |^2 dx \notag \\																			&\cdot \left( \int_{0}^{T} \int_{S^1} | (-\Delta)^{1/2} u |^2 dx dt + \frac{1}{R^2} \int_{0}^{T} \int_{S^1} | (-\Delta)^{1/4}u |^2 dx dt \right).
	\end{align}
\end{lem}

\begin{proof}
	The key observation lies in the following estimate: In \cite{wettstein}, we used the rather crude estimate:
	\begin{align}
		&\int_{0}^{T} \int_{S^1} \Big{|} P.V. \int_{S^{1}} (-\Delta)^{1/4} u(y) \frac{\varphi(x) - \varphi(y)}{| x- y|^{3/2}} dy \Big{|}^2 dx dt \notag \\
		\label{tobeused}
		&\lesssim \int_{0}^{T} \int_{S^1} | (-\Delta)^{1/4} u(y) |^2 \frac{1}{| x- y |^{1/2}} dy \cdot \int_{S^1} \frac{| \varphi(x) - \varphi(y) |^2}{| x-y |^{5/2}} dy dx dt \\
		\label{crude}
		&\lesssim \frac{1}{R^{2}} \int_{0}^{T} \int_{S^1} | (-\Delta)^{1/4} u(y) |^2 dy dt,
	\end{align}
	where $\varphi$ is a cut-off function on some subset $B_{R}(x_0)$, $x_0 \in S^1$. In \cite{wettstein}, we then obtained \eqref{struweest02old} by summing for a suitable covering by balls with finite-intersection property the terms \eqref{crude}. Instead of using \eqref{crude}, we will now use \eqref{tobeused} and obtain a more precise estimate. For each fixed value $x \in S^1$, we have then a sum:
	$$\sum_{j \in I} \frac{| \varphi_{j}(x) - \varphi_{j}(y) |^2}{| x-y|^{5/2}},$$
	which we want to estimate in order to establish \eqref{struweest02} using \eqref{tobeused} and summation over a suitable covering. Here, $\varphi_{j}$ are the corresponding cut-offs to a suitable covering, i.e. they are supported in balls of radius $R$ ($\varphi_j$ being equal to $1$ on the ball with same center and radius $3/4 R$) with the property that every point is contained in at most $3$ of these balls. In fact, the covering should be as in \cite{wettsteinsphere}. Now, if $\varphi_{j}(x) \neq 0$, we use the estimate:
	$$\frac{| \varphi_{j}(x) - \varphi_{j}(y) |^2}{| x-y|^{5/2}} \leq \| \nabla \varphi \|_{L^{\infty}} \frac{1}{| x-y |^{1/2}} \lesssim \frac{1}{R^{2}} \cdot \frac{1}{| x-y |^{1/2}}$$
	Notice that $\varphi(x) \neq 0$ only holds true for finitely many $j$, this number being independent of $R$, so by integrating over $S^{1}$ and exploiting the integrability of $1/|x-y|^{1/2}$ on $S^1$, we deduce that the contribution of these terms may be bounded by $1/R^2$.\\
	
	Next, we have to consider all terms with $\varphi_j (x) = 0$. By choice of the covering in \cite{wettsteinsphere}, it is clear that then:
	$$| x-y | \geq \delta R,$$
	for some $\delta > 0$ independent of $R$.
	Indeed, the cover may be chosen in such a way that for $\delta > 0$ small and independent of $R$, we have that for $x \in S^1$, $B_{\delta R}(x)$ lies in one of the balls of the covering. Then only finitely many have non-empty intersection with this ball around $x$ and thus all others must satisfy 
	$$| x - y | \geq \delta R,$$
	for $y$ in the remaining balls. Taking next the ball which gets closest to $x$ among all with empty intersection with $B_{\delta R}(x)$, we have that again only finitely many have non-empty intersection with this one, all others satisfy 
	$$| x-y | \geq (\delta +1)R,$$
	for $y$ these balls. Iterating such an argument and observing that the number of intersecting balls may be controlled independent of $R$, we arrive at the estimate ultimately required.
	Thus, by integration of these summands, we obtain a sum of the form:
	$$\sum_{j \in \mathbb{N}_{0}} \frac{1}{R^{3/2}} \cdot \frac{3}{(j + \delta)^{3/2}} \lesssim \frac{1}{R^{3/2}} \lesssim \frac{1}{R^2}, \quad \forall R \in ]0,1/2[$$
	Indeed, observe that the covering may be chosen in such a way that at each point, at most $3$ of the balls intersect. Noting that we may select balls and describe the distance between $x$ and the corresponding balls in terms of $(j + \delta) R$, the statement becomes apparent. Then by integrating $1/|x-y|^{5/2}$ explicitly, we obtain the sum above. Combining both contributions, we get the improved estimate \eqref{struweest02} by arguing as in \cite{wettstein}.
\end{proof}

Such a result also allows for a slightly more refined version of Lemma 3.19 in \cite{wettsteinsphere}:

\begin{lem}
\label{lemmaestpaper}
	There exists $\varepsilon_1 > 0$ such that for any $u \in H^{1}([0,T] \times S^1) \cap L^{\infty}([0,T]; H^{1/2}(S^1))$ solving:
	$$\partial_{t} u + (-\Delta)^{1/2} u \perp T_u N \quad \text{ in } \mathcal{D}'([0,T] \times S^1)$$ 
	with values in $N$ and any $R < 1/2$, there holds:
			\begin{equation}
				\int_{0}^{T} \int_{S^1} | \nabla u |^2 dx dt \leq C E(u_0) \left( 1 + \frac{T}{R^2} \right),
			\end{equation}
			with $C$ independent of $u, T, R$, provided $\varepsilon(R) < \varepsilon_{1}$. Here, $u(0, \cdot) = u_0 \in H^{1/2}(S^1;N)$ is the initial value.
\end{lem}

The proof is as in \cite{wettstein} or \cite{struwe1}, the only change lies in the application of Lemma \ref{improvedlemmajw} instead of Lemma 3.16 in \cite{wettsteinsphere}. This improved version will be crucial in the blow-up procedure, as it will enable us to deduce that the $H^1$-energy is bounded and thus leads to a good solution after extracting a weakly convergent subsequence, since we have now an appropriate scaling-behaviour of time and space variable.

\subsubsection{Lower Bound for Local Gagliardo Seminorms}

Next, we would like to establish a connection between the concentration condition \eqref{concentrationcond} at blow-up points and the Gagliardo-seminorms at the same points. The intuition behind the estimate is that whenever $1/2$-Dirichlet energy concentrates close to a point, then also the localised Gagliardo seminorm around the same point should concentrate, just as it is the case for the harmonic gradient flow in some sense (the statement is however tautological in this case, as the energy is already local). Due to the non-local nature, however, contributions from further away may still be significant, forcing us to include a bigger domain in the estimate of the seminorm than in the $1/2$-energy to avoid concentration in "neck regions" that we would otherwise not account for. The key connection is the following:

\begin{prop}
\label{prop1}
	Let $\varepsilon >0$ be given and $N$ big enough depending on $\varepsilon$. Assume that $u \in H^{1/2}(S^1)$ with $| u | \leq 1$ and such that:
	$$\int_{B_{R}(x_0)} | (-\Delta)^{1/4} u |^2 dx \geq \varepsilon,$$
	for some $R < 2^{-N-1}$. Then there is a $\delta > 0$ depending only on $n$ and $\varepsilon$, such that:
	$$\int_{B_{2^{N}R}(x_0)} \int_{B_{2^{N}R}(x_0)} \frac{| u(x) - u(y) |^2}{| x-y |^{2}} dy dx \geq \delta.$$
\end{prop}

In the proof, we shall clarify the necessary requirement for $N$. Also, the same proof continues to hold for arbitrary bounded $u$ with $\delta$ depending also on $\| u \|_{L^{\infty}}$.

\begin{proof}
	Firstly, we observe that the independence of $R$ and $x_0$ of $\delta$ may be obtained by rescaling and rotations, possibly after using stereographic projection. So we do not have to worry about such dependencies.\\
	
	Let us argue by contradiction: Assume the statement was wrong, then there exists a sequence $u_n \in H^{1/2}(S^1)$ of bounded functions, such that:
	$$\int_{B_{R}(x_0)} | (-\Delta)^{1/4} u_n |^2 dx \geq \varepsilon; \quad \int_{B_{2^{N}R}(x_0)} \int_{B_{2^{N}R}(x_0)} \frac{| u_n(x) - u_n(y) |^2}{| x-y |^{2}} dy dx < \frac{1}{n}$$
	In particular, we have (up to modifying the $u_n$ by a constant and extracting a subsequence):
	$$u_n \to u \quad \text{ in } H^{1/2}(B_{2^N R}(x_0))$$
	As seen in \cite{hitchhiker}, we may extend the $u_n \in H^{1/2}(B_{2^N R}(x_0))$ to $v_n \in H^{1/2}(S^1)$ which are still bounded by a common multiple of $1$ and such that:
	$$\| v_n \|_{H^{1/2}(S^1)} \lesssim \| u_n \|_{H^{1/2}(B_{2^N R}(x_0))} \to 0,$$
	which also shows:
	$$\lim_{n \to \infty} \int_{S^1} | (-\Delta)^{1/4} v_n |^2 dx = 0.$$
	Thus, to arrive at a contradiction, we just need to show:
	$$\liminf_{n \to \infty} \int_{B_{R}(x_0)} | (-\Delta)^{1/4} (u_n - v_n) |^2 dx < \varepsilon$$
	This can be easily obtained by observing:
	\begin{align}
		&\int_{B_{R}(x_0)} | (-\Delta)^{1/4} (u_n - v_n) |^2 dx \notag \\
		&\leq \int_{B_{R}(x_0)} \left( \int_{B_{2^{N} R}(x_0)^{c}} \frac{| u_n(y) - v_{n}(y)|}{| x-y |^{3/2}} dy \right)^{2} dx \notag \\
		&\lesssim \int_{B_{R}(x_0)} \left( \int_{B_{2^{N} R}(x_0)^{c}} \frac{1}{| x-y |^{3/2}} dy \right)^{2} dx \notag \\
		&\leq \int_{B_{R}(x_0)} \frac{1}{| x \mp 2^{N} R|} dx \notag \\
		&\lesssim | \log \left( 1 - 2^{-N} \right) | \lesssim 2^{-N} < \varepsilon,
	\end{align}
	provided $N$ was chosen sufficiently large at the beginning, depending on $\varepsilon$. Thus the required statement follows, as this contradicts our assumptions and thus provides the desired contradiction.
\end{proof}

The key feature of Proposition \ref{prop1} lies in the fact that it connects the localised (but still non-local) Gagliardo-seminorms to the concentration of energy. The power of $2$ that appears is due to the non-linearity and ensures that "not too much" energy is lost by restricting to balls. Ensuring that energy is stored in balls of sufficiently small radius is crucial to obtain half-harmonic maps in the limit.

\subsubsection{Bubbling-Analysis}

Having proved Lemma \ref{improvedlemmajw} as well as Proposition \ref{prop1}, we are now able to study the bubbling process in points where energy accumulates. The analysis is inspired by \cite{struwe1}, but has to take care of the non-local behaviour associated with the fractional Laplacian:

\begin{thm}
\label{thminpaper}
	Let $u$ be a solution as in Theorem \ref{mainresjw} and let $x_0 \in S^1$ be a point, such that:
	\begin{equation}
	\label{concentrationinpt}
		\limsup_{t \to T} \int_{B_{R}(x_0)} | (-\Delta)^{1/4} u |^2 dx \geq \varepsilon_1, \quad \forall R > 0,
	\end{equation}
	where $\varepsilon_1 > 0$ is as in \cite[Lemma 4.10]{wettstein}. Then there exists a half-harmonic map $v: \R \to S^{n-1}$, such that:
	\begin{equation}
		u_n \to v \quad \text{ weakly in } H^{1}(\R) \text{ and strongly in } H^{1/2}(\R),
	\end{equation}
	where $u_n$ is a suitable rescaling and translation of $u$.
\end{thm}

As stated in the introduction, an analogous result holds for any closed $N$ instead of $S^{n-1}$ as target manifold, up to some technical changes in the formulas. Additionally, we highlight that \eqref{concentrationcond} implies \eqref{concentrationinpt} at a suitable point by choosing subsequences. Therefore, Theorem \ref{thminpaper} actually concerns the behaviour of functions at the critical time in Theorem \ref{mainresjw}. It should be noted that the number of points $x_0$ satisfying \eqref{concentrationinpt} is finite due to the limited amount of energy available, so these points may not accumulate.

\begin{proof}
	Let us argue along the lines of \cite[Theorem 4.3]{struwe1}. The key idea is to rescale $u$ on subintervals of $[0, T[$ and apply the results in Lemma \ref{improvedlemmajw} and Proposition \ref{prop1} to deduce convergence. Let us always assume that $N$ is chosen large enough to allow for $\varepsilon = \varepsilon_1/2$ in Proposition \ref{prop1} and take $\delta > 0$ to be the associated lower bound for the Gagliardo seminorms.\\
	
	We now define rescalings as follows: For each $R > 0$, we have:
	\begin{equation}
		\varphi_{R}: \R \to S^1 \simeq \R / \Z \simeq [-\pi;\pi[,
	\end{equation}
	with the properties:
	$$\varphi_{R}(x) = R^2 x, \quad \forall x \in [-\frac{2^{N}}{R}, \frac{2^{N}}{R}]; \quad | \varphi_{R}'(x) | \leq R^2,$$
	and:
	$$\lim_{x \to \pm \infty} \varphi_{R}(x) = \pm \pi.$$
	The existence of such a function is clear.\\
	
	By \eqref{concentrationinpt} and choosing points $(t_n, x_n) \in [0,T[ \times S^1$ as in \cite{struwe1} with $t_n \to T, x_n \to x_0$ and such that:
	$$E_{R_n}(u(t_n, \cdot), x_n) = \varepsilon_1 = \sup_{0 < t \leq t_n, x \in B_{r}(x_0)} E_{R_n}(u(t, \cdot), x),$$
	where $R_n \to 0$ and $r > 0$ is chosen small enough that no other point with the property \eqref{concentrationinpt} is contained in $B_r(x_0)$. We shall now define:
	\begin{equation}
		u_n: [-\gamma, 0] \times \R \to S^{n-1}, \quad u_n(t,x) := u(t_n + R_{n}^{2}t, x_n + \varphi_{R_n}(x))
	\end{equation}
	Here, $\gamma > 0$ (using Lemma \ref{lemmaestpaper}) is chosen in such a way to ensure:
	$$E_{2R_{n}}(u(t); x_n) \geq \varepsilon_1/2, \quad \forall t \in [t_n - \gamma R_{n}^2, t_n].$$
	See also Lemma 4.9 in \cite{wettstein} for a justification of this fact and compare this with \cite{struwe1}. To define $x_n + \varphi_{R_n}(x)$, we may use the periodicity of $u$ in the space-variable. The key properties of these functions are their boundedness properties. For example, we have:
	\begin{align}
		&\int_{-\gamma}^{0} \int_{\R} | \nabla u_n(t,x) |^2 dx dt \notag \\
		&= \frac{1}{R_{n}^2} \int_{t_n - \gamma R_{n}^2}^{t_n} \int_{S^1} | \nabla u(s,y) |^2 | \varphi_{R_n}'(\varphi_{R_{n}}^{-1}(y)) |^2 | (\varphi_{R_{n}}^{-1})'(y) | dy ds \notag \\
		&= \frac{1}{R_{n}^2} \int_{t_n - \gamma R_{n}^2}^{t_n} \int_{S^1} | \nabla u(s,y) |^2 | \varphi_{R_n}'(\varphi_{R_{n}}^{-1}(y)) | dy ds \notag \\
		&\leq \int_{t_n - \gamma R_{n}^2}^{t_n} \int_{S^1} | \nabla u(s,y) |^2 dy ds \lesssim E(u_0),
	\end{align}
	where we used Lemma \ref{lemmaestpaper} as well as the choice of points $(t_n, x_n)$ as above. Notice that the chain rule is employed at one point to simplify the expression. Similarily:
	\begin{align}
		&\int_{-\gamma}^{0} \int_{B_{2^{N}/R_{n}}(0)} | \partial_t u_n(t,x) |^2 dx dt \notag \\
		&= R_{n}^2 \int_{t_n - \gamma R_{n}^2}^{t_n} \int_{B_{2^{N} R_{n}}(x_n)} | \partial_{t} u(s,y) |^2 | (\varphi_{R_{n}}^{-1})'(y) | dy dt \notag \\
		&= \int_{t_n - \gamma R_{n}^2}^{t_n} \int_{B_{2^{N} R_{n}}(x_n)} | \partial_{t} u(s,y) |^2 dy dt \notag \\
		&\leq \int_{t_n - \gamma R_{n}^2}^{t_n} \int_{S^1} | \partial_{t} u(s,y) |^2 dy dt \lesssim E(u_0)
	\end{align}
	One may now extract convergent subsequences as in \cite{struwe1}. Thus, we end up with sequences $u_n(\tau_n, \cdot)$ which converge weakly in $H^{1}(S^1)$ and strongly in $H^{1/2}(\R)$ to $v \in H^{1}(\R)$. Choosing the subsequence to be pointwise convergent a.e., we may even deduce:
	$$v \in S^{n-1} \quad \text{a.e.}$$
	Furthermore, Proposition \ref{prop1} shows, thanks to the concentration of energy, that:
	$$\delta \leq \int_{B_{2^{N} R_n}(x_n)} \int_{B_{2^{N} R_{n}}(x_n)} \frac{| u(t,x) - u(t,y) |^2}{| x-y |^2} dy dx,$$
	for all $t \in [t_n - \gamma R_n^2, t_n]$. This also shows:
	$$\delta \leq \int_{B_{2^{N}/R_n}(x_n)} \int_{B_{2^{N}/R_{n}}(x_n)} \frac{| u_n(\tau_n,x) - u_n(\tau_n,y) |^2}{| x-y |^2} dy dx$$
	Thus, by passing to the limit as $n \to \infty$:
	$$E_{1/2}(v) \geq \delta > 0,$$
	and so $v$ may not be constant. It remains to check that $v$ is actually half-harmonic. This is however an immediate consequence of the original equation:
	$$\partial_t u + (-\Delta)^{1/2} u = u | d_{1/2} u |^2$$
	Namely, since for $\tau_n$, we have:
	$$\partial_t u(\tau_n) \to 0,$$
	as $n \to \infty$ in $L^{2}_{loc}(\R)$, it remains to prove convergence of the other terms. Namely, we have for any $\varphi \in C^{\infty}_{c}(\R)$:
	\begin{align}
		&\int_{\R} \int_{\R} d_{1/2} v(x,y) d_{1/2} \varphi(x,y) \frac{dy dx}{| x-y |} \notag \\
		&= \lim_{n \to \infty} \int_{\R} \int_{\R} d_{1/2} u_n(\tau_n) (x,y) d_{1/2} \varphi(x,y) \frac{dy dx}{| x-y |} \notag \\
		&= \lim_{n \to \infty} \int_{B_{2^N / R_n}} \int_{B_{2^N / R_n}} d_{1/2} u_n(\tau_n) (x,y) d_{1/2} \varphi(x,y) \frac{dy dx}{| x-y |} \notag \\
		&= \lim_{n \to \infty} \int_{B_{2^N R_n}(x_n)} \int_{B_{2^N R_n}(x_n)} d_{1/2} u(\tau_n) (x,y) d_{1/2} \left( \varphi \circ \varphi_{R_n}^{-1} \right) (x,y) \frac{dy dx}{| x-y |} \notag \\
		&= \lim_{n \to \infty} \int_{S^1} \int_{S^1} d_{1/2} u(\tau_n) (x,y) d_{1/2} \left( \varphi \circ \varphi_{R_n}^{-1} \right) (x,y) \frac{dy dx}{| x-y |} \notag \\
		&= \lim_{n \to \infty} \left( \int_{S^1} -\partial_{t} u \cdot \varphi \circ \varphi_{R_n}^{-1} dx + \int_{S^1} u(\tau_n) | d_{1/2} u(\tau_n) |^2 \cdot \varphi \circ \varphi_{R_n}^{-1} dx \right) \notag \\
		&= \lim_{n \to \infty} \int_{S^1} u(\tau_n) | d_{1/2} u(\tau_n) |^2 \cdot \varphi \circ \varphi_{R_n}^{-1} dx \notag \\
		&= \lim_{n \to \infty} \int_{B_{2^N R_n}(x_n)} u(\tau_n) \int_{B_{2^{N} R_n}(x_n)} | d_{1/2} u(\tau_n)(x,y) |^2 \frac{dy}{| x-y |} \cdot \varphi \circ \varphi_{R_n}^{-1} dx \notag \\
		&=  \lim_{n \to \infty} \int_{B_{2^N / R_n}(x_n)} u_n(\tau_n) \int_{B_{2^{N} / R_n}(x_n)} | d_{1/2} u_n(\tau_n)(x,y) |^2 \frac{dy}{| x-y |} \cdot \varphi dx \notag \\
		&= \int_{\R} v | d_{1/2} v |^2 \varphi,
	\end{align}
	which is the desired equation. Notice that throughout the computations, we used several times that appropriate terms may be omitted due to the boundedness of $u_n(\tau_n)$ and $v$, leading to omissions of parts of the domain of integration, switching between the distance function on $S^1$ and $\R$ and similar terms. A crucial observation is that $\varphi$ is supported on a subdomain of $B_{2^{N} / R_n}$ for $R_n$ sufficiently small, so the estimates have good bounds everywhere, if $n$ goes to $\infty$.
	So we are done, since $v$ solves the half-harmonic map equation and thus is actually smooth, see \cite{daliopigati}. In particular, $v$ may be regarded as a $1/2$-harmonic map after composition with the stereographic projection.
\end{proof}

\subsection{Existence of Global Solutions} \label{global}

Finally, we have all the necessary tools at our disposal to tackle the global existence problem in full generality. The main idea will be that one is easily able to extend solutions on a finite time-interval by using convergence properties as $t$ goes to the critical time. A direct argument shows that the extension by gluing a solution at the critical time for appropriate initial data will give a global solution after at most finitely many such extensions.

\subsubsection{Proof by "Gluing"}

Let us show that we may extend a solution $u: [0,T[ \times S^1 \to S^{n-1}$ to be a weak solution on a slightly bigger time interval. This may be done by first observing that due to the monotone decay of energy:
\begin{equation}
	E_{1/2}(u(t)) \leq E_{1/2}(u_0) < + \infty
\end{equation}
Therefore, we may deduce that for an appropriate sequence $u(t_n) \to v \in H^{1/2}(S^1)$ with $t_n \to T$. Moreover, since $u \in H^{1}([0,T[;L^{2}(S^1))$, we must have convergence:
\begin{equation}
\label{limitinl2ofsols}
	\lim_{t \to T} u(t) = v \quad \text{ in } L^{2}(S^1),
\end{equation}
due to a standard continuity argument. This also shows uniqueness of $v$ independent of any choice of sequence $t_n \to T$.

Next, we want to estimate the $1/2$-energy of $v$. To do this, let us assume that there is just one bubbling point $x_0$ at time $T$ (the general case follows analogously, losing energy in finitely many points). Then we have:
\begin{align}
	E_{1/2}(v)	&= \int_{S^1} \int_{S^1} \frac{| v(x) - v(y) |^2}{| x-y |^2} dy dx \notag \\
			&= \lim_{r \to 0} \int_{S^{1} \setminus B_{r}(x_0)} \int_{S^{1} \setminus B_{r}(x_0)} \frac{| v(x) - v(y) |^2}{| x-y |^2} dy dx \notag \\
			&= \lim_{r \to 0} \liminf_{n \to \infty} \int_{S^{1} \setminus B_{r}(x_0)} \int_{S^{1} \setminus B_{r}(x_0)} \frac{| u(t_n, x) - u(t_n, y) |^2}{| x-y |^2} dy dx \notag \\
			&\leq \liminf_{n \to \infty} E_{1/2}(u_n) - \varepsilon_0 = \lim_{t \to T} E_{1/2}(u(t)) - \varepsilon_0,
\end{align}
where $\varepsilon_0$ denotes a quantum of energy that is concentrated close to $x_0$. As $\varepsilon_0$ is independent of $u$ and $T$, we deduce that bubbling may only occur in finitely many points, as the $1/2$-energy is decreasing and bounded from below by $0$. Thus, we do not have to worry about accumulations of blow-up points.\\

One concludes now by extending the solution $u$ after $T$ by the main existence result in \cite{wettstein}, Theorem \ref{mainresjw}. The fact that we have obtained a weak solution is easily verified by a direct computation based on the $L^{2}$-convergence of $u(t)$ as $t \to T$, thus establishing the desired global existence result.
Indeed, we assume that $u:[0,+\infty[ \times S^1 \to S^{n-1}$ bubbles at time $T = 1$, the general case with finitely many times in which bubbling occur follows completely analogously. Let $\varphi \in C^{\infty}_{c}(]0,\infty[ \times S^1)$, since we know that the equation holds true for sufficiently small times. Then we have:
\begin{align}
	&\int_{0}^{\infty} \int_{S^1} \partial_{t} u \cdot \varphi dx dt + \int_{0}^{\infty} \int_{S^1} (-\Delta)^{1/2} u \cdot \varphi dx dt \notag \\
	&= - \int_{0}^{\infty} \int_{S^1} u \cdot \partial_{t} \varphi dx dt + \int_{0}^{\infty} \int_{S^1} (-\Delta)^{1/4} u \cdot (-\Delta)^{1/4} \varphi dx dt \notag \\
	&= - \int_{0}^{1} \int_{S^1} u \cdot \partial_{t} \varphi dx dt + \int_{0}^{1} \int_{S^1} (-\Delta)^{1/4} u \cdot (-\Delta)^{1/4} \varphi dx dt \notag \\
	&- \int_{1}^{\infty} \int_{S^1} u \cdot \partial_{t} \varphi dx dt + \int_{1}^{\infty} \int_{S^1} (-\Delta)^{1/4} u \cdot (-\Delta)^{1/4} \varphi dx dt \notag \\
	&= \int_{0}^{1} \int_{S^1} \partial_{t} u \cdot \varphi dx dt + \int_{0}^{1} \int_{S^1} (-\Delta)^{1/4} u \cdot (-\Delta)^{1/4} \varphi dx dt - \int_{S^1} u(1,x) \varphi(1,x) dx \notag \\
	&+ \int_{1}^{\infty} \int_{S^1} \partial_{t} u \cdot \varphi dx dt + \int_{1}^{\infty} \int_{S^1} (-\Delta)^{1/4} u \cdot (-\Delta)^{1/4} \varphi dx dt + \int_{S^1} u(1,x) \varphi(1,x) dx \notag \\
	&= \int_{0}^{1} u | d_{1/2} u |^2 \varphi dx dt - \int_{S^1} u(1,x) \varphi(1,x) dx + \int_{S^1} u(1,x) \varphi(1,x) dx + \int_{1}^{\infty} u | d_{1/2} u |^2 \varphi dx dt \notag \\
	&= \int_{0}^{\infty} u | d_{1/2} u |^2 \varphi dx dt,
\end{align}
which proves the fact that $u$ extended as explained yields a global weak solution. The first line equation is just the distributional formulation, later on we use integration by parts on $[0,\tilde{t}]$ and taking limits $\tilde{t} \to T$. Naturally, similar limits are taken for $[\tilde{t}, \infty[$. Observe that the boundary terms at time $T = 1$ appear due to the previous discussion of convergence in $L^2$ and by the boundary value properties of the extension, see Theorem \ref{mainresjw}. We highlight that $u(1,x)$ is defined for the extended solution to be that limit of the $u(t,x)$ in $L^2$ and weak limit in $H^{1/2}$, as $t \to 1$, see \eqref{limitinl2ofsols}. Iterating this procedure finitely many times provides therefore a global weak solution.

In conclusion, we have the following, since the argument only superficially relies on $N = S^{n-1}$:

\begin{thm}
	Let $u_0 \in H^{1/2}(S^1;N)$, then there exists a weak solution with non-increasing $1/2$-Dirichlet energy:
	$$u: [0, +\infty [ \times S^1 \to N,$$
	with $u \in L^{\infty}([0,+\infty[; H^{1/2}(S^1)) \cap H^{1}([0,+\infty[;L^{2}(S^1))$ such that, except for finitely many times $0 < T_1 < \ldots < T_{n} < T_{n+1} := + \infty$, the function $u$ is smooth:
	$$u \in C^{\infty}(]T_k, T_{k+1}[; N), \quad \forall k = 1, \ldots n.$$
	Moreover, we may bound the number $n = n(u_0)$ as follows:
	$$n(u_0) \leq \frac{E(u_0)}{\varepsilon_0},$$
	where $\varepsilon_0 > 0$ is the minimum amount of $1/2$-energy a non-constant, half-harmonic map with values in $N$ must possess.
\end{thm}

A first uniqueness statement may also be derived from the results in \cite{wettstein}. However, it should be noted that uniqueness among energy class solution (weak solutions) cannot be proven by our previous arguments and thus requires further investigations. Finally, the existence of finite time bubbling is still unresolved, so the result above provides potentially a suitable regularity statement at bubbling points to help understand obstructions for bubbling or build examples in future work.

\subsubsection{Proof by Variational Arguments}

In this section, we derive an alternative proof of the global weak existence of solutions to the half-harmonic gradient flow using techniques from Calculus of Variations similar to \cite{audrito}. This approach does lead to existence of solutions, however, it leaves open many questions regarding the properties of the solution, most importantly regarding monotonicity of the $1/2$-Dirichlet energy. In particular, if the solutions constructed do not have monotonically decaying energy, then the solution provides an example of non-uniqueness of solutions to the half-harmonic map equation.\\

The definition of the energy follows \cite{audrito}. Let $\varepsilon > 0$ be any positive real number. We define the following space of functions for $s \in ]0,1[$ and $1 < p < + \infty$:
$$\mathcal{V}^{s,p} :=  H^{1}([0,+\infty[; L^{2}(S^1;\R^n)) \cap L^{2}_{loc}([0,+\infty[;W^{s,p}(S^1;\R^n)),$$
and use this definition to introduce for any $u_0 \in W^{s,p}(S^1;N)$, where $N$ is a closed submanifold in $\R^n$:
\begin{equation}
\label{solspace}
	\mathcal{U}^{s,p}(u_0) := \big{\{} u \in \mathcal{V}^{s,p}\ \big{|}\ u(t,x) \in N \text{ a.e.},
%	 \| u(t) \|_{H^{1/2}(S^1)} \leq \| u_0 \|_{H^{1/2}(S^1)}, 
	 u(0) = u_0 \big{\}}
\end{equation}
Comparing with \cite{schisirewang}, the space \eqref{solspace} actually coincides with space in which the solutions constructed there exist. Moreover, we define the following family of energies:
\begin{equation}
\label{defenergy}
	\mathcal{E}^{s,p}_{\varepsilon}(u) := \int_{0}^{+\infty} \int_{S^1} e^{-t/\varepsilon} \left( \varepsilon \cdot | \partial_{t} u(t,x) |^2 + \frac{2}{p} \cdot \int_{S^1}  \left| \frac{u(t,x) - u(t,y)}{| x-y |^s} \right|^{p} \frac{dy}{| x-y |} \right) dx dt,
\end{equation}
for any $u \in \mathcal{U}^{s,p}(u_0)$. One notices that the energy is indeed well-defined and finite in this case. An obvious member of $\mathcal{U}^{s,p}(u_0)$ is the following map:
$$u(t,x) := u_0(x),$$
and this shows:
\begin{equation}
\label{simpleminimiserbound}
	\inf_{u \in \mathcal{U}^{s,p}(u_0)} \mathcal{E}^{s,p}_{\varepsilon}(u) \leq 2 E_{s,p}(u_0) \cdot \int_{0}^{\infty} e^{-t/\varepsilon} dt = 2 \varepsilon \cdot E_{s,p}(u_0),
\end{equation}
where we use the definition of $E_{s,p}$ as in \cite{schisirewang}. Thus, we immeidately see that if $(u_\varepsilon)_{\varepsilon \in ]0,1[}$ is a sequence of minimizers, then the energies will become arbitrarily small. Additionally, existence of minimizers can easily be proven by the direct method.

Defining $v(t,x) := u(\varepsilon t, x)$, we see:
\begin{align}
	\mathcal{E}^{s,p}_{\varepsilon}(u)	&= \int_{0}^{+\infty} \int_{S^1} e^{-t/\varepsilon} \left( \varepsilon \cdot | \partial_{t} u(t,x) |^2 + \frac{2}{p} \cdot \int_{S^1}  \left| \frac{u(t,x) - u(t,y)}{| x-y |^s} \right|^{p} \frac{dy}{| x-y |} \right) dx dt \notag \\
								&= \int_{0}^{+\infty} \int_{S^1} \varepsilon e^{-s} \left( \varepsilon \cdot | \partial_{t} u(\varepsilon s,x) |^2 + \frac{2}{p} \cdot \int_{S^1}  \left| \frac{u(\varepsilon s,x) - u(\varepsilon s,y)}{| x-y |^s} \right|^{p} \frac{dy}{| x-y |} \right) dx ds \notag \\
								&= \int_{0}^{+\infty} \int_{S^1} e^{-s} \left(  | \partial_{t} v(s,x) |^2 + \frac{2\varepsilon}{p} \cdot \int_{S^1}  \left| \frac{v(s,x) - v(s,y)}{| x-y |^s} \right|^{p} \frac{dy}{| x-y |} \right) dx ds \notag \\
								&=: \mathcal{J}^{s,p}_{\varepsilon}(v)
\end{align}
Notice that $v$ still lies in $\mathcal{U}^{s,p}(u_0)$ and that by computation above, we know that minimising $\mathcal{E}_{\varepsilon}$ and minimising $\mathcal{J}_{\varepsilon}$ is equivalent respecting the reparametrisation in time.\\

Let us now compute the Euler-Lagrange equation for $\mathcal{E}^{s,p}_{\varepsilon}$:
\begin{lem}
	The Euler-Lagrange equation for minimisers $u \in \mathcal{U}^{s,p}(u_0)$ of $\mathcal{E}^{s,p}_{\varepsilon}$ can be stated as:
	\begin{equation}
	\label{eulerlagrangeeps}
		- \varepsilon \partial_{t}^2 u(t,x) + \partial_{t} u(t,x) + \div_{s} \left( |d_{s}u(t,x,y) |^{p-2} d_{s} u(t,x,y) \right) \perp T_{u} N, \quad \text{ in } \mathcal{D}'(]0,+\infty[ \times S^1)
	\end{equation}
\end{lem}

\begin{proof}
	We take the competitors:
	$$u_{\delta}(t,x) := \pi(u + \delta \varphi),$$
	where $\delta \in \R$ and $\varphi \in C^{\infty}_{c}(]0,\infty[ \times S^1;\R^n)$. Moreover, $\pi$ denotes the closest point projection onto $N$.\\
	
	If $u$ is a minimizer, then:
	$$0 = \frac{d}{d\delta} \mathcal{E}^{s,p}_{\varepsilon}(u_{\delta}) \Big{|}_{\delta = 0}$$
	Using the explicit formula \eqref{defenergy} for the energy (observe that $u_{\delta}$ lies in the correct space for every $\delta \in \R$ sufficiently small), one can differentiate immediately (we use $d_0 u(t,x,y) = u(t,x) - u(t,y)$ to simplify the terms):
	$$0 = \int_{0}^{+\infty} \int_{S^1} e^{-t/\varepsilon} \left( 2\varepsilon \partial_{t} u \cdot \partial_{t} \left( d\pi(u) \varphi \right) + 2 \int_{S^1} \frac{| u(t,x) - u(t,y) |^{p-2}}{| x-y |^{1 + sp}} d_{0}u(t,x,y) \cdot d_{0} \left( d\pi(u) \varphi \right)(t,x,y) dy  \right) dx dt$$
	If we choose $\psi(t,x) = e^{-t/\varepsilon} \varphi(t,x)$, then:
	$$0 = \int_{0}^{+\infty} \int_{S^1} \varepsilon \partial_{t} u \cdot \partial_{t} \left( d\pi(u) \psi \right) + \partial_{t} u \cdot d\pi(u) \psi + \int_{S^1} \frac{| u(t,x) - u(t,y) |^{p-2}}{| x-y |^{1 + sp}} d_{0}u(t,x,y) \cdot d_{0} \left( d\pi(u) \psi \right)(t,x,y) dy dx dt$$
	So the Euler-Lagrange equation is equivalent to:
	$$- \varepsilon \partial_{t}^2 u(t,x) + \partial_{t} u(t,x) + \div_{s} \left( |d_{s}u(t,x,y) |^{p-2} d_{s} u(t,x,y) \right) \perp T_{u} N, \quad \text{ in } \mathcal{D}'(]0,+\infty[ \times S^1),$$
	i.e. up to the term involving the second derivative in time direction we recognise the fractional harmonic gradient flow. This proves \eqref{eulerlagrangeeps}.
\end{proof}
In particular, if $s = 1/2, p = 2$, we find the same equation as in \cite{wettstein}, up to the second order derivative in $t$. This is also the case we shall restrict our attention to for now (writing $\mathcal{J}_{\varepsilon}$ instead of $\mathcal{J}^{1/2,2}_{\varepsilon}$), the general case for arbitrary fractional harmonic flows may be treated in a completely analogous way, also extending the existence result in \cite{schisirewang} in a wider setting.\\

The ideas to complete the proof then are very similar to \cite{audrito}. Namely, one may define completely analogously:
\begin{align}
	\label{defI}
	I(t)	&:= \int_{S^1} | \partial_{t} v(t,x) |^2 dx \\
	\label{defR}
	R(t)	&:= \varepsilon \cdot  \int_{S^1} | d_{1/2} v(t) |(x)^2 dx \\
	\label{defE}
	E(t)	&:= e^{t} \int_{t}^{\infty} e^{-s} \left( I(s) + R(s) \right) ds
\end{align}
It is easily observed that for miniizers $v$, we have $I, R \in L^{1}_{loc}([0,\infty[)$ and $e^{-s} (I(s) + R(s)) \in L^{1}([0,\infty[)$. Additionally, $E \in W^{1,1}_{loc}(]0,\infty[) \cap C^{0}([0,\infty[)$ as well as:
$$E' = E - I - R \quad \text{ in } \mathcal{D}'(]0, \infty[)$$
The proof of the following lemma is an immediate adaption of the technique in \cite{audrito}:

\begin{lem}
	Assume $v$ is a minimizer of $\mathcal{J}_{\varepsilon}$. Then:
	\begin{equation}
	\label{derivE}
		E'(t) = - 2I(t), \quad \text{ in } \mathcal{D}'(]0,\infty[)
	\end{equation}
\end{lem}

The proof relies on suitable choices of reparametrisations in time for $v$ and then using minimality of $v$. Ultimately, this allows us to show:

\begin{lem}
	For $v$ a minimizer of $\mathcal{J}_{\varepsilon}$, we have:
	\begin{equation}
	\label{estparttforminimizer}
		\int_{0}^{\infty} | \partial_{t} v(t,x) |^2 dx dt \leq C \varepsilon,
	\end{equation}
	as well as for any $t \geq 0$:
	\begin{equation}
	\label{estd1/2forminimizer}
		\int_{t}^{t+1} \int_{S^1} | d_{1/2} u |^2(t,x) dx dt \leq C,
	\end{equation}
	for some constant $C > 0$, depending on $u_0$, but not $\varepsilon$ or $v$.
\end{lem}

\begin{proof}
	$E(t)$ is necessarily non-increasing due to $I(t) \geq 0$, therefore:
	$$E(t) \leq E(0) = \mathcal{J}_{\varepsilon}(v), \forall t \geq 0$$
	Additionally, for any given $t$, we know:
	$$\int_{0}^{t} I(s) ds = \frac{1}{2} \int_{0}^{t} E'(s) ds = \frac{1}{2} \left( E(0) - E(t) \right) \leq \frac{1}{2} E(0) \leq \frac{C}{2} \varepsilon,$$
	by using \eqref{simpleminimiserbound}. Letting $t \to \infty$ proves \eqref{estparttforminimizer} by using \eqref{defI}.\\
	
	The remaining part of the proof requires us to use \eqref{defR} as well as:
	\begin{align}
		\int_{t}^{t+1} R(s) ds	&= e^{t+1} \int_{t}^{t+1} e^{-s} R(s) ds \notag \\
						&\leq e^{t+1} \int_{t}^{t+1} e^{-s} \left( I(s) + R(s) \right) ds \\
						&\leq e \cdot E(t) \leq Ce \cdot \varepsilon, 
	\end{align}
	again relying on \eqref{simpleminimiserbound} and the bound on $E(t)$ established above.
\end{proof}

Thus, to obtain a solution of the half-harmonic gradient flow (which follows thanks to \eqref{eulerlagrangeeps} after letting $\varepsilon \to 0$), one now just has to rescale the minimizer $v$ back to $u$ and use the following uniform bounds to extract weakly convergent subsequences. Thus, we are done, as we may extract further subsequences converging almost surely pointwise, ensuring that the limiting function assumes values only in $N$.

\end{document}